\documentclass[11pt,a4paper,twoside]{article}
\usepackage{a4wide}
\usepackage{amsmath,amssymb,amsthm}
\usepackage[latin1]{inputenc}
\usepackage[nonamebreak, round]{natbib}
\usepackage{graphicx,subfigure,color}
\usepackage{algorithm, algorithmicx, algpseudocode}
\usepackage[bookmarks=false,colorlinks=false,linkcolor=blue]{hyperref}

\algrenewcommand{\algorithmiccomment}[1]{\hskip1em \textit{//#1}}

\def\E{{\mathbb E}}

\def\P{{\mathbb P}}
\def\Q{{\mathbb Q}}
\def\R{{\mathbb R}}
\def\N{{\mathbb N}}

\def \cE{\mathcal{E}}
\def \cF{\mathcal{F}}

\def \cD{\mathcal{D}}
\def \cC{\mathcal{C}}
\def \cP{\mathcal{P}}
\def \cL{\mathcal{L}}
\def \cN{\mathcal{N}}

\def \cT{\mathcal{T}}
\def \probinv{{\overline{\mathbb{Q}}}}
\def\ovt{\overline{t}}
\def\expp#1{\mathop {\mathrm{e}^{ #1}}}

\def\Var{\mathop{{\mathbb V}\mathrm{ar}}}

\def\ind#1{{\mathbf 1}_{\{#1\}}}

\def\cite{\citet}

\def\keywords#1{{\bf Keywords~:~}#1}

\numberwithin{equation}{section}
\theoremstyle{plain}
\newtheorem{thm}{Theorem}
\newtheorem{prop}[thm]{Proposition}
\newtheorem{lemma}[thm]{Lemma}
\newtheorem{corollary}[thm]{Corollary}
\newtheorem{remark}[thm]{Remark}

\newtheoremstyle{hypo@style}{\topsep}{15pt}{\itshape}{0pt}{\bfseries}{}{5pt}
{\thmname{#1}\thmnumber{ #2}\thmnote{ {\mdseries (#3)}}}
\theoremstyle{hypo@style}
\newtheorem{hypo}{Hypothesis}

\emergencystretch=\hsize
\begin{document}

\title{Stochastic Local Intensity Loss Models with Interacting Particle Systems}

\author{
Aurélien Alfonsi \footnote{Email:alfonsi@cermics.enpc.fr}\\
Université Paris-Est, CERMICS,\\
 Project Team MathRisk ENPC-INRIA-UMLV, Ecole des
Ponts,\\ 6-8 avenue Blaise Pascal, 77455 Marne-la-Vallée, France\vspace*{0.2cm}\\
Céline Labart \footnote{Email:celine.labart@univ-savoie.fr}\\
Laboratoire de Mathématiques, CNRS UMR 5127\\
Université de Savoie, Campus Scientifique\\
73376 Le Bourget du Lac, France\vspace*{0.2cm}\\
 Jérôme Lelong \footnote{Email:jerome.lelong@imag.fr}\\
 Univ. Grenoble Alpes, Laboratoire Jean Kuntzmann,  \\
  BP 53, 38041 Grenoble, Cedex 09, France}

\date{\today}

\maketitle

\begin{abstract} 
  It is well-known from the work of~\cite{Schonbucher} that the marginal laws of a loss process can be matched by a unit increasing time inhomogeneous Markov process, whose deterministic jump intensity is called local intensity. The Stochastic Local Intensity (SLI) models such as the one proposed by~\cite{Arnsdorf} allow to get a stochastic jump intensity while keeping the same marginal laws. These models involve a non-linear SDE with jumps. The first contribution of this paper is to prove the existence and uniqueness of such processes. This is made by means of an interacting particle system, whose convergence rate towards the non-linear SDE is analyzed. Second, this approach provides a powerful way to compute pathwise expectations with the SLI model: we show that the computational cost is roughly the same as a crude Monte-Carlo algorithm for standard SDEs. 

  \noindent  \keywords{ Stochastic local intensity model, Interacting particle systems, Loss
  modelling, Credit derivatives, Monte-Carlo Algorithm, Fokker-Planck equation, Martingale problem.} \\
  \textbf{AMS classification (2010): 91G40, 91G60, 60H35} 
\end{abstract}

{{\bf Acknowledgments:} We would like to thank Benjamin Jourdain (Université Paris-Est) for fruitful discussions having stimulated this research. We would like to thank INRIA Paris Rocquencourt for allowing us to run our numerical tests on the convergence rate on their RIOC cluster. A.~Alfonsi acknowledges the support of the ``Chaire Risques Financiers'' of Fondation du Risque. C.~Labart and J.~Lelong acknowledge the support of the Finance for Energy Market Research Centre, www.fime-lab.org.}

\newpage

\section{Introduction}

In equity modeling, a major concern is to get a model that fits 
option data. It is well-known from the work of~\cite{Dupire} that
basically European options can be exactly calibrated by using local volatility
models~$\sigma(t,x)$. However, local volatility models are known to have some
inadequacy to describe real markets. To get richer dynamics, Stochastic Local
Volatility (SLV) models have been introduced (see~\cite{Alexander} or~\cite{Piterbarg}) and consider the following 
dynamics for the stock under a risk-neutral probability:
$$dS_t=rS_tdt +f(Y_t) \eta(t,S_t) S_t dW_t,$$
where $Y_t$ is an adapted stochastic process. Typically, $(Y_t,t\ge 0)$ is
assumed to solve an autonomous one dimensional~SDE whose Brownian motion may
be correlated with~$W$.
From the work of~\cite{Gyongy}, we know that under mild
assumptions, the following choice
$$ \eta(t,x)= \frac{\sigma(t,x)}{\sqrt{\E[f(Y_t)^2|S_t=x]}}$$
ensures that $S_t$ has the same marginal laws as the local volatility model
with~$\sigma(t,x)$, which automatically gives the calibration to European
option prices. This leads to the following non-linear SDE
$$dS_t=rS_tdt +\frac{f(Y_t)} {\sqrt{\E[f(Y_t)^2|S_t]}}\sigma(t,S_t) S_t
dW_t.$$
Here, we stress that the law of $(Y_t,S_t)$ steps into the diffusion term. Unless
for trivial choices of~$Y_t$ and despite some attempts (\cite{Abergel}), getting the existence and uniqueness of solutions for  this kind of SDE remains an open problem. Also, from a numerical perspective, the simulation of SLV models is not easy, precisely because of the computation of the conditional expectation.

In this paper, we propose to tackle a very analogous problem arising in
credit risk modeling. In all the paper, we will work under a risk-neutral 
probabilistic filtered space $(\Omega,(\cF_t)_{t\ge 0}, \cF, \P)$. As usual,
$\cF_t$ is the $\sigma$-field describing all the events that can occur
before time $t$ and $\cF$ describes all the events. We consider
$M\in \N^*$~defaultable entities (for example, $M=125$ for the iTraxx). We assume that all the recovery rates are
deterministic and equal to $1- L_{GD}$ for all the firms within the
basket. The loss process $(L_t,t\ge0)$ is given by
\begin{align*}
  L_t=\frac{L_{GD}}{M}X_t, \; \forall t \ge 0
\end{align*}
where $X_t$ is the number of defaults up to time $t$. Clearly, $X$ takes values in
$\cL_M:=\{0,\dots,M\}$. Thanks to the assumption of deterministic recovery
rates, and under the assumption of deterministic short interest rates, it is
well known that CDO tranche prices only depend on the marginal laws of the
loss process (see for example Remark 3.2.1 in~\cite{Alfonsi}). Let us
assume then for a while that we have found marginal laws
$(P(X_t=k),k\in\cL_M)_{t\in [0,T]}$ which perfectly fit CDO tranche
prices up to maturity~$T>0$. Then, under some mild assumptions, we know from
\cite{Schonbucher} that there exists a non-homogeneous Markov
chain with only unit increments which exactly matches these
marginal laws. Somehow, this result plays the same role for the loss as Dupire's
result for the stock.

The loss model obtained with a non-homogeneous unit-increasing Markov chain is
known in the literature as local intensity model. It is fully described by the 
local intensity $\lambda:\R_+ \times \cL_M \rightarrow \R_+$, which gives the
instantaneous rate of having one more default in the basket. In the sequel, we
will assume that the local intensity $\lambda:\R_+ \times \cL_M
\rightarrow \R_+$ has been calibrated to market data and perfectly matches
Index and CDO tranche prices. We make the following assumptions:
\begin{itemize}
\item $\forall x \in \cL_M , t \in \R \mapsto \lambda(t,x)$ is a càdlàg function,
\item $\forall t\ge 0, \ \lambda(t,M)=0$.
\end{itemize}
In particular, we have $\overline{\lambda}=\max_{x\in \cL_M}
\sup_{t\in[0,T]}\lambda(t,x) < \infty$.
In this setting, the instantaneous jump rate at time $t$ from~$x$ to~$x+1$ is given by
$\lambda(t-,x)$. Thus, the local intensity model corresponds to a time-inhomogeneous Markov chain making unit jumps with this rate. 

One may like however get richer dynamics than the ones given by the local
intensity model. Then, we can proceed in the same way as for the stochastic
volatility models. Let us consider $(Y_t)_{t\ge 0}$ a general
$(\cF_t)$-adapted càdlàg real process, a function $f:\R \rightarrow \R_+$ and a
function $\eta:\R_+\times \cL_M \rightarrow \R_+$ satisfying the same
assumptions as~$\lambda$. We assume that the default
counting process $(X_t,t \ge 0)$ has jumps of size~$1$ with the rate
\begin{align*}
  \eta(t-,X_{t-})f(Y_{t-}).
\end{align*} By analogy with the equity, we name this kind of model a
Stochastic Local Intensity (SLI) model. Then, it is known (see
\cite{ContMinca}) that the local intensity model with the Local Intensity (LI)
$\eta(t-,x)\E[f(Y_{t-})|X_{t-}=x]$ has the same marginal laws as~$X_t$. Thus,
the SLI model will be automatically calibrated to CDO tranche prices if one
takes:
$$\forall t>0, x \in \cL_M, \
\eta(t-,x)=\frac{\lambda(t-,x)}{\E[f(Y_{t-})|X_{t-}=x]}. $$ This approach has
been used in the literature by~\cite{Arnsdorf}, and in a slightly different
way by~\cite{Lopatin}. However, up to our knowledge there is no proof in the
literature of the existence nor uniqueness of such a dynamics.

The first scope of this paper is to solve this problem. At this stage,
we need to make our framework precise. We assume through the
paper that: 
\begin{equation}\label{assump_f} f:\R\mapsto\R \text{ is continuous, s.t. }
  \forall x \in \R, \ 0< \underline{f}\le f(x) \le \overline{f}<\infty.
\end{equation}
We assume that  the probability space $(\Omega, \mathcal{F}, \P)$ contains a
 standard Brownian motion $(W_t,t\ge 0)$, a sequence of independent uniform
 random variables $(U^k)_{k\in \N}$, and a sequence $(E^n)_{n\in \N}$ of
 independent exponential random variables with
 parameter~$\frac{\overline{\lambda}\overline{f}}{\underline{f}}$. We set
 $T^k=\sum_{n=1}^kE^n$ for $k\in \N^*$. The random variables
 $(T^k,U^k)_{k\in \N^*}$ will enable us to define a non homogeneous Poisson point process
with jump intensity $\eta(t-,X_{t-})f(Y_{t-})$. 
We are interested in studying the following two problems in which we assume that
$Y_t$ is a process with values either in~$\N$ (discrete
case) or in~$\R$ (continuous case). In the discrete case, we are interested in
finding a predictable process $(X_t,Y_t)_{t\ge 0}$ such that
\begin{equation}\label{discrete}
\begin{cases} X_t=x_0+\sum_{k, T^k\le t} \mathbf{1}_{\left\{U^k
   \le\frac{\underline{f}}{\overline{\lambda}\overline{f}}
   \frac{f(Y_{T^k-})\lambda(T^k-,X_{T^k-})}{\E[f(Y_{T^k-})|X_{T^k-}]}\right\}}\\
 Y_0=y_0, \text{ and for each }k\ge 0,\ (Y_t, t \in [T_k,T_{k+1}) ) \text{ is a continuous time Markov chain}\\
\text{with transition rate } \mu^{X_t}_{ij}=\mu^{X_{T_k}}_{ij}.
\end{cases}
\end{equation}
For the sake of simplicity, we consider in the discrete setting that $X$ and $Y$ do not jump together almost surely. 

In the continuous case, the corresponding problem is to solve the following
stochastic differential equation:
\begin{equation}\label{continuous}
\begin{cases}
  X_t=x_0+\sum_{k, T^k\le t} \mathbf{1}_{\left\{U^k \le\frac{\underline{f}}{\overline{\lambda}\overline{f}} \frac{f(Y_{T^k-})\lambda(T^k-,X_{T^k-})}{\E[f(Y_{T^k-})|X_{T^k-}]}\right\}}\\
  Y_t=y_0+\int_0^t b(s,X_s,Y_s)ds +\int_0^t \sigma (s,X_s,Y_s)dW_s + \int_0^t \gamma(s-,X_{s-},Y_{s-})dX_s,
\end{cases}
\end{equation}
for $x_0 \in \cL_M,y_0 \in \R$ and given real functions $b$, $\sigma$
and~$\gamma$. This framework embeds in particular the dynamics suggested by
\cite{Arnsdorf} and~\cite{Lopatin}.  Under some rather mild hypotheses on
$\mu_{ij}$, $b$, $\sigma$ and~$\gamma$, which will be specified in the
corresponding sections, we will show that the above two equations admit a
unique solution. In the discrete case, we are able to show that the
corresponding Fokker-Planck equation has a unique solution. This can be
achieved by writing the Fokker-Planck equation as an ODE which can be studied
directly. In the continuous case, this approach can hardly been extended: the
Fokker-Planck equation leads to a non trivial PDE. Instead, we solve this
problem by introducing an interacting particle system. This technique is known
to be powerful for this type of non linear problems (see~\cite{Sznitman}
or~\cite{Meleard}).

The second scope of this paper is to provide a way to compute prices under SLI
models. Indeed, interacting particle systems are not only theoretical tools to
prove existence and uniqueness results for such equations. They give a very
smart way to simulate these processes, therefore enabling us to run Monte-Carlo
algorithms. This approach has been recently used by~\cite{GHL} for Stochastic
Local Volatility models. For the Stochastic Local Intensity models considered
in this work, the conditional expectation is much simpler to handle. This
enables us to get theoretical results on the convergence and also simplifies
the implementation.  In fact, we show in our case under some assumptions that
the rate of convergence to estimate expectations is in~$O(1/N^\alpha)$ for any
$\alpha<1/2$, where $N$ is the number of particles. On our numerical
experiments, we even observe on several examples a convergence which is similar
to the one of the Central Limit Theorem, which is rather usual for Interacting
Particle Systems. Besides, we show that we can simulate the interacting
particle system with a computational cost in $O(DN)$, where $D$ is the number
of time steps for the discretization of the SDE on~$Y$. Thus, the computational
cost is roughly the same as a crude Monte-Carlo algorithm for standard SDEs
with $N$ samples. \\

The paper is organized as follows. First, we study the case where $Y$ has
discrete values; this framework enables us to settle the problem and solve
it by rather elementary tools. This part is independent from the rest of the
paper. Second, by means of a particle system approach, we investigate the case
where $Y$ is real valued jump diffusion. Finally, we carry out numerical
simulations highlighting the relevance of the particle system technique to
compute pathwise expectation of the Process~\eqref{continuous}.

\section{The SLI model when~$Y$ takes discrete values}

The goal of this section is to prove the existence of a process $(X_t,Y_t)_{t \ge 0}$ satisfying~\eqref{discrete}. Unlike the continuous case~\eqref{continuous}, we can get this result by elementary means, without resorting to an interacting particle system. To do so, we write the Fokker-Planck equation associated to the process $(X_t,Y_t)_{t \ge
  0}$, which should be satisfied by $\P(X_t=i,Y_t=j)$ for  $(i,j) \in \cL_M\times \N$. We have
$$(\mathcal{E})
\left\{
\begin{array}{lcl}
    \partial_t p(t,i,j) & = & \sum_{k \neq j}
    \mu^i_{kj} p(t,i,k)+\ind{i \ge
      1}\frac{\lambda(t,i-1)}{\varphi_p(t,i-1)}f(j)p(t,i-1,j)\\
     & &-\left(\frac{\lambda(t,i)}{\varphi_p(t,i)}f(j)\ind{i\le
      M-1}-\mu^i_{jj}\right)p(t,i,j)
   \\
    p(0,i,j)&  = & 0 \;\forall (i,j)\neq (x_0,y_0), \\
    p(0,x_0,y_0)& = & 1.
\end{array}\right.
 $$
where $p$ is a function from $\R_+ \times \cL_M \times \N$ to $\R$ and
\begin{align*}
  \varphi_p(t,i)=\frac{\sum_{j=0}^{\infty}f(j)p(t,i,j)}{\sum_{j=0}^{\infty}p(t,i,j)}.
\end{align*}
 If we manage to prove that the Fokker-Planck equation admits a
unique solution $p$ such that
\begin{align}
  &\forall i,j \in \cL_M \times \N, \forall t \ge 0, \;p(t,i,j) \ge 0, \label{eq1}\\
& \forall t \ge 0 \;\sum_{i=0}^M \sum_{j=0}^{\infty} p(t,i,j)=1,\label{eq2}
\end{align}
then we will get that the law of a process $(X_t,Y_t)_{t \ge 0}$ satisfying~\eqref{discrete} is unique. Besides, we will also get the existence of such a process. It is easy to check that a continuous Markov chain $(X_t,Y_t)_{t\ge 0}$ starting from~$(x_0,y_0)$ with transition rate matrix $$\tilde{\mu}_{(i_1,j_1),(i_2,j_2)}(t)=\mathbf{1}_{j_1=j_2}\left(\mathbf{1}_{i_2=i_1+1}- \mathbf{1}_{i_2=i_1}\right)\frac{f(j_1)\lambda(t,i_1)}{\varphi_p(t,i_1)}  +\mathbf{1}_{i_1=i_2}\mu^{i_1}_{j_1j_2}, \ i_1,i_2\in \cL_M, j_1,j_2\in \N, $$ where $p$ is  the solution of~$(\mathcal{E})$ satisfies~\eqref{discrete}. In fact, the Fokker-Planck equation of this process
$$
\left\{
\begin{array}{lcl}
    \partial_t q(t,i,j) & = & \sum_{k \neq j}
    \mu^i_{kj} q(t,i,k)+\ind{i \ge
      1}\frac{\lambda(t,i-1)}{\varphi_p(t,i-1)}f(j)q(t,i-1,j)\\
     & &-\left(\frac{\lambda(t,i)}{\varphi_p(t,i)}f(j)\ind{i\le
      M-1}-\mu^i_{jj}\right)q(t,i,j)
   \\
    q(0,i,j)&  = & 0 \;\forall (i,j)\neq (x_0,y_0), \\
    q(0,x_0,y_0)& = & 1.
\end{array}\right.
 $$
is linear and clearly solved by~$p$, which gives $q\equiv p$.

\subsection{Assumptions and notations}

 In this part, we assume that the transition rates satisfy the
following hypothesis.
\begin{hypo}\label{hypo:discrete} The intensity matrices $(\mu^k_{ij})_{i,j\ge 0}$ for $k\in \cL_M$ satisfy the following conditions:
\begin{itemize}
\item $\forall k \in \cL_M$, $\forall i,j \in \N\times \N$ such that $i \neq
  j$ $\mu^k_{ij}\ge 0$,
\item $\forall k \in \cL_M$, $\forall i \in \N$ $ \mu^k_{ii}\le 0$,
\item $\forall k \in \cL_M$, $\forall i \in \N$ $\sum_{j=0}^{\infty}
  \mu^k_{ij}=0$ (then $\sum_{j=0,j
    \neq i}^{\infty} \mu^k_{ij}=-\mu^k_{ii} $).\\
\end{itemize} Moreover, we assume $\forall k \in \cL_M$, $ \sup_{i \in \N}
|\mu^k_{ii}| <\infty$.
\end{hypo}
 
We also introduce specific notations used in the discrete case.
\begin{itemize}
  \item We define the set $E$ as the set of real sequences indexed by $i \in \cL_M$ and $j \in
  \N$: $$E:=\{u=(u^i_j)_{0\le i \le M, 0 \le j}: u^i_j \in
    \R \},$$
$E_+:=\{u=(u^i_j)_{0\le i \le M, 0 \le j}: u^i_j \in
    \R_+ \}$ and $E^*_+:=\{u=(u^i_j)_{0\le i \le M, 0 \le j}: u^i_j \in
    \R_+^* \}$.
\item For $u \in E$, we set $|u|=\sum_{i=0}^M \sum_{j=0}^{\infty}  |u^i_j|$.
\end{itemize}

\subsection{Solving the Fokker-Planck equation} To be more concise, we rewrite
the Fokker-Planck Equation ($\cE$) and Conditions $(\ref{eq1})-(\ref{eq2})$ by
using a sequence of functions. Let $P:=(P^i_j)_{0\le i \le M, j \ge 0}$ denote
a sequence such that each $P^i_j$ is a function from $\R_+$ to $\R$. Solving
$(\cE)$ under Constraints $(\ref{eq1})-(\ref{eq2})$ boils down to solving
 $$ (\cE')
\begin{cases}
    P'(t) & =\Psi(t,P(t)),\\
    P(0)&=P_0
\end{cases}
$$
under the constraints $\forall t \ge 0, \; P(t) \ge 0$ and $|P(t)|=1$. The
sequence $P_0$ is such that $(P_0)^i_j=0$ for $(i,j) \neq (x_0,y_0)$ and
$(P_0)^{x_0}_{y_0}=1$. $\Psi$ is an application from $\R_+ \times
E_+ \rightarrow E$ given by
\begin{align*}
  (\Psi(t,x))^i_j=&\sum_{k \ge 0} \mu^i_{kj} x^i_k+\ind{i \ge
    1}\frac{\lambda(t,i-1)}{\varphi(x,i-1)}f(j)x^{i-1}_j-
  \ind{i \le M-1}\left(\frac{\lambda(t,i)}{\varphi(x,i)}f(j)\right)x^i_j
\end{align*}
where $\varphi(x,i):=\frac{\sum_{l=0}^{\infty}f(l)x^{i}_l}{\sum_{l=0}^{\infty}x^{i}_l}$.

\begin{remark} $\Psi$ is defined without ambiguity on $E_+^*$. When $x\in E_+$, difficulties
  may arise when for some fixed $i$, $x^i_j=0$ $\forall j$. In this case, we still have 
  $\displaystyle   \lim_{z \in E_+^*,z \rightarrow x} \frac{z^i_j}{\varphi(z,i)}=0$ since $\underline{f} \le \varphi(z,i)\le \overline{f}$ for $z \in E_+^*$. Thus, we can extend
  $\Psi$ by continuity on $E_+$.
\end{remark}
 We aim at solving $(\cE)$ in the set of summable sequences (compatible with Condition
(\ref{eq2})) and get the following result.

\begin{thm}\label{thm_FP} Equation $(\mathcal{E}')$ admits a unique solution on $\R_+$
  satisfying $\forall t \ge 0, \; P(t) \ge 0$ and $|P(t)|=1$.
\end{thm}

 To do so, we first focus on the following differential equations: 
 $$ (\cE'')
\begin{cases}
    P'(t) & =\Psi(t,(P(t))_+),\\
    P(0)&=P_0.
\end{cases}
$$

\begin{prop}\label{prop_FP}
  Equation $(\cE'')$ admits a unique solution on $\R_+$. Moreover, the
  solution satisfies  $\forall t \ge 0, \; P(t) \ge 0$ and $|P(t)|=1$.
\end{prop}
The proof of this proposition consists in first showing the Lipschitz property which gives the existence and uniqueness of~$P$ and then proving that $\forall t \ge 0, \; P(t) \ge 0$ and $|P(t)|=1$. This proof is postponed to Appendix~\ref{App_prop_FP}.

Then, the proof of Theorem~\ref{thm_FP} becomes obvious. The unique solution~$P$ of $(\mathcal{E}'')$ clearly solves $(\mathcal{E}')$, and any solution~$Q$ of $(\mathcal{E}')$ such that
  $\forall t \ge 0, \; Q(t) \ge 0$ and $|Q(t)|=1$ also solves
  $(\mathcal{E}'')$ and thus coincides with~$P$.


\section{The SLI model when~$Y$ is real valued.}

\subsection{Setting and main results}
We are interested in proving the existence of a process $(X_t,Y_t)_t$ solving
the stochastic differential equation~\eqref{continuous}. More precisely, we
will consider the following stochastic differential equation,
\begin{equation}\label{continuous2}
\begin{cases}
  X_t=X_0+\sum_{k, T^k\le t} \mathbf{1}_{\left\{U^k \le\frac{\underline{f}}{\overline{\lambda}\overline{f}} \frac{f(Y_{T^k-})\lambda(T^k-,X_{T^k-})}{\E[f(Y_{T^k-})|X_{T^k-}]}\right\}}\\
  Y_t=Y_0+\int_0^t b(s,X_s,Y_s)ds +\int_0^t \sigma (s,X_s,Y_s)dW_s + \int_0^t \gamma(s-,X_{s-},Y_{s-})dX_s,
\end{cases}
\end{equation}
with (possibly) random initial condition $(X_0,Y_0)$ such that
$\E[|Y_0|^m]<\infty$ for any $m\in \N$. We will denote in the sequel~$\cL_{init}$
the probability law of~$(X_0,Y_0)$ under~$\P$. To get existence and uniqueness
results for~\eqref{continuous2}, we will make
the following assumption on the coefficients.

\begin{hypo}\label{hypo:continuous}
\begin{enumerate}
\item The functions $b,\ \sigma,\ \gamma:\R_+\times \cL_M \times \R \rightarrow \R$ are
  measurable,  with sub-linear growth with respect to~$y$:
   $$\forall T>0,\exists C_T>0, \  \forall t \in [0,T], x \in \cL_M, y\in \R, |b(t,x,y)|+|\sigma(t,x,y)|+|\gamma(t,x,y)|\le C_T(1+|y|).$$

 \item The functions $b(t,x,y)$ and $\sigma(t,x,y)$ are such that for any~$x
   \in \cL_M$, $y_0\in \R$, there exists a unique strong solution for the SDE
  $$Y_t=y_0+\int_0^tb(s,x,Y_s)ds +\int_0^t \sigma(s,x,Y_s)dW_s, \ t\ge 0.$$
  This property holds if we assume for example that:
   $$\forall T>0,\exists C_T>0, \  \forall t \in [0,T], x \in \cL_M,
   \begin{cases}
     |b(t,x,y)-b(t,x,y')|\le C_T |y-y'| \\
     |\sigma(t,x,y)-\sigma(t,x,y')|\le C_T\sqrt{|y-y'|}.
   \end{cases}
   $$

 \item For any $x\in \cL_M$, $(t,y)\mapsto \gamma(t,x,y)$ is càdlàg with
   respect to~$t$ and continuous with respect to~$y$, i.e. 
  $\gamma(t,x,y)=\lim_{s>t,s\rightarrow t, z\rightarrow y}\gamma(s,x,z)$ and
  $\lim_{s<t,s\rightarrow t, z\rightarrow y}\gamma(s,x,z)$ exists ans is
  denoted by $\gamma(t-,x,y)$. 
\end{enumerate}
\end{hypo}

To prove the strong existence and uniqueness of a process $(X,Y)$ solving~\eqref{continuous2}, we will first need to prove a weak existence and uniqueness result. To do so, we introduce the Martingale Problem associated with~\eqref{continuous2}. We denote by $\cD([0,T],\R)$ the set of càdlàg real valued functions and consider:
 $$E=\{(x(t),y(t))_{t\in[0,T]}, s.t.\; y\in
 \cD([0,T],\R) \text{ and } x \text{ is a nondecr. càdlàg function with values in }
\cL_M \}.$$
This path space is endowed with the usual Skorokhod topology for càdlàg processes, and with the
associated Borelian $\sigma$-algebra. We denote by $\cP(E)$ the set of probability measures on~$E$. We are
looking for a probability measure $\probinv \in \cP(E)$ such that $\cL_{init}$ is the probability
law of $(X_0,Y_0)$ under~$\probinv$ and, for any $\phi \in
\cC^{0,2}(\R^2,\R)$,
\begin{eqnarray}\label{pb_mg}
  \phi(X_t,Y_t)-\phi(X_0,Y_0)-&& \hspace{-0.5cm} \displaystyle \int_0^t \Big\{\frac{\lambda(u,X_u)f(Y_u)}{\E_\probinv[f(Y_u)|X_{u}]}[\phi(X_u+1,Y_u+\gamma(u,X_u,Y_u))-\phi(X_u,Y_u)]
 \nonumber \\ &&\hspace{-1.3cm}
+  b(u,X_u,Y_u)\partial_y\phi(X_u,Y_u)+\frac{1}{2}\sigma^2(u,X_u,Y_u)\partial^2_y\phi(X_u,Y_u)
  \Big\}du 
\end{eqnarray}
is a martingale with respect to the filtration $\cF_t=\sigma((X_u,Y_u),u \le
t)$ satisfying the usual conditions. Here, $X_t$ and $Y_t$ stand for the coordinate applications:
$$X_t:\begin{array}[t]{l} E \rightarrow \cL_M \\ (x,y) \mapsto x(t)
\end{array}
 \text{ and } Y_t:\begin{array}[t]{l} E \rightarrow \R \\ (x,y) \mapsto y(t),
\end{array}$$
and $\E_\probinv$ denotes the expectation under the
probability measure~$\probinv \in \mathcal{P}(E)$. Similarly, for any $\Q\in \cP(E)$, we denote by  $\E_\Q$ the expectation under the probability measure~$\Q$ while $\E$ simply denotes the expectation under the original probability measure~$\P$.\\

The following Theorem is the main result of the paper.
\begin{thm}\label{thm_continuous}
 We assume that Hypothesis~\ref{hypo:continuous} holds and consider a
 $\cF_0$-measurable initial condition $(X_0,Y_0)$ such that $\forall m \in \N,
 \E[|Y_0|^m]<\infty$. Then, there exists a unique probability
 measure~$\probinv\in \cP(E)$ solving the Martingale Problem~\eqref{pb_mg}.
 Besides, there exists a unique strong solution $(X_t,Y_t)_{t\ge 0}$
 to Equation~\eqref{continuous2}.
\end{thm}

To prove Theorem \ref{thm_continuous}, we need the following basic result on standard SDEs with jumps.  This is an easy consequence of Hypothesis~\ref{hypo:continuous}. For the sake of completeness, we give its proof in Appendix~\ref{app_proof_unicite_eds}.
\begin{prop}\label{unicite_eds_sauts} For~$\Q \in \mathcal{P}(E)$, we set for $t\in[0,T], x\in \cL_M$
 $$\varphi^{\Q}(t,x)= \frac{
  \E_{\Q}[f(Y_{t})\ind{X_{t}=x}] }{\Q(X_t=x)}, \text{ when }\Q(X_t=x)>0\text{ and }\varphi^{\Q}(t,x)=\underline{f}\text{ otherwise.}$$
  Let Hypothesis~\ref{hypo:continuous} hold. Then, for any $(x_0,y_0)\in \cL_M \times \R$, there exists a unique strong solution $(X_t,Y_t)_{t\in[0,T]}$ to the following SDE with jumps:
\begin{equation}\begin{cases}\label{EDS_avec_sauts}
  X_t=x_0+\sum_{k, T^k\le t} \mathbf{1}_{\left\{U^k \le\frac{\underline{f}}{\overline{\lambda}\overline{f}} \frac{f(Y_{T^k-})\lambda(T^k-,X_{T^k-})}{\varphi^{\Q}(T^k-,X_{T^k-})} \right\}}\\
  Y_t=y_0+\int_0^t b(s,X_s,Y_s)ds +\int_0^t \sigma (s,X_s,Y_s)dW_s + \int_0^t \gamma(s-,X_{s-},Y_{s-})dX_s.
\end{cases}
\end{equation}
\end{prop}

\begin{proof}[Proof of Theorem \ref{thm_continuous}]
  The proof of Theorem~\ref{thm_continuous} is split in three main steps.
  \begin{itemize}
    \item First, we show the existence of a probability measure $\probinv \in
      \cP(E)$ solving the Martingale Problem~\eqref{pb_mg}. This result is
      obtained by considering the associated interacting particle system: we
      show that each particle converges in law, and that any probability
      measure in the support of the limiting law solves the Martingale Problem.
      This is done in Section~\ref{section_existence}.
    \item Second, we show the uniqueness of the probability measure $\probinv
      \in \cP(E)$ solving~\eqref{pb_mg}.  To do so, we introduce a function
      $\Psi:\cP(E) \rightarrow \cP(E)$ defined as follows. Let $(X_0,Y_0)$ be a
      random variable distributed according to~$\cL_{init}$ under~$\P$. Then,
      we know from Proposition~\ref{unicite_eds_sauts} that there exists a
      unique process $(X^{\Q}_t,Y^{\Q}_t)_{t\in[0,T]}$ solving
      \begin{equation}\label{defXqYq}\begin{cases}
        X_t^{\Q}=X_0+\sum_{k, T^k\le t} \mathbf{1}_{\left\{U^k \le\frac{\underline{f}}{\overline{\lambda}\overline{f}} \frac{f(Y_{T^k-}^{\Q})\lambda(T^k-,X_{T^k-}^{\Q})}{\varphi^{\Q}(T^k-,X^{\Q}_{T^k-})} \right\}}\\
        Y_t^{\Q}=Y_0+\int_0^t b(s,X_s^{\Q},Y_s^{\Q})ds +\int_0^t \sigma (s,X_s^{\Q},Y_s^{\Q})dW_s + \int_0^t \gamma(s-,X_{s-}^{\Q},Y_{s-}^{\Q})dX^{\Q}_s,
      \end{cases}
    \end{equation}
    and we define 
    $$\Psi(\Q)=law((X^{\Q}_t,Y^{\Q}_t)_{t\in[0,T]}) \in \mathcal{P}(E).$$
    As we will see in Section~\ref{section_unicite}, $\Psi$ (or more precisely
    $\Psi$ iterated $k$-times) is a contraction mapping for the variation
    norm. Combining this result with the following Lemma gives the uniqueness of the
    probability measure solving~\eqref{pb_mg}.

    \begin{lemma}\label{lemme_probinv}
      Let $\probinv \in \cP(E)$. We have $\Psi(\probinv)=\probinv$ if and only if~$\probinv$ solves the Martingale Problem~\eqref{pb_mg}. In this case, $(X_t^\probinv,Y_t^\probinv)_{t\in[0,T]}$ solves Equation~\eqref{continuous2}.
    \end{lemma}

  \item Then, the existence and uniqueness of~$\probinv$ satisfying~$\Psi(\probinv)=\probinv$ and Lemma~\ref{lemme_probinv} automatically give the strong existence and uniqueness of $(X,Y)$ satisfying~\eqref{continuous2} since we necessarily have  $\E[f(Y_{T^k-})|X_{T^k-}]=\E_\probinv[f(Y_{T^k-})|X_{T^k-}]$ and thus $(X_t,Y_t)_{t\in[0,T]}=(X_t^\probinv,Y_t^\probinv)_{t\in[0,T]}$.
\end{itemize}
\end{proof}

\begin{proof}[Proof of Lemma \ref{lemme_probinv}] The direct implication is
  clear since $\Psi(\probinv)=\probinv$ gives that
  $\E[f(Y_t^{\probinv})|X_t^{\probinv}]=\varphi^{\probinv}(t,X_{t}^{\probinv})$. Thus,
  $(X_t^\probinv,Y_t^\probinv)_{t\in[0,T]}$ solves~\eqref{continuous2} and
  in particular $\probinv$ solves the  Martingale Problem~\eqref{pb_mg}.
  
  Conversely, let us assume that $\probinv$ solves the Martingale Problem~\eqref{pb_mg}. We know from Proposition~\ref{unicite_eds_sauts} that strong uniqueness holds for the SDE $(X^{\probinv}_t,Y^{\probinv}_t)_{t\in[0,T]}$. From \cite[Theorem II$_{13}$]{Lepeltier}, we know that strong uniqueness implies weak uniqueness, which precisely gives $\Psi(\probinv)=\probinv$ since $\probinv$ solves the Martingale Problem~\eqref{pb_mg} and therefore the Martingale problem associated with~\eqref{defXqYq}. In particular, we have  $\E[f(Y_t^{\probinv})|X_t^{\probinv}]=\varphi^{\probinv}(t,X_{t}^{\probinv})$, and $(X_t^\probinv,Y_t^\probinv)_{t\in[0,T]}$ solves~\eqref{continuous2}.
\end{proof}




\subsection{The interacting particle system}\label{sec_ips}

We assume that the probability space $(\Omega, \mathcal{F}, \P)$  carries all
the random variables used below.  Now, we set up the particle system related to
the Martingale Problem~\eqref{pb_mg}. In the following, $N$ will denote the
number of particles, $(X^i_0,Y^i_0),i\in \N$ are independent random variables
following the law~$\cL_{init}$ under~$\P$, and $(W^i_t,t\ge 0),i\in \N$ are
independent standard Brownian motions. We build an interacting particle system
$((X^{i,N}_t,Y^{i,N}_t),t\ge 0)_{1 \le i\le N}$ with the following features.
For $i=1,\dots,N$,  $(X^{i,N}_t,t\ge 0)$ is a Poisson process with intensity:
\begin{align}\label{eq3}
  \frac{\lambda(t-,X^{i,N}_{t-})f(Y^{i,N}_{t-}) \sum_{j=1}^N \ind{X^{j,N}_{t-}=X^{i,N}_{t-}} }{\sum_{j=1}^N
    f(Y^{j,N}_{t-})\ind{X^{j,N}_{t-}=X^{i,N}_{t-}}},
\end{align}
and $Y^{i,N}_t$ solves the following equation:

\begin{align}\label{eq4}
  Y^{i,N}_t=Y^i_0+ \int_0^t b(s,X^{i,N}_s,Y^{i,N}_s) ds +\int_0^t
  \sigma(s,X^{i,N}_s,Y^{i,N}_s) dW^i_s+\int_0^t
  \gamma(s-,X^{i,N}_{s-},Y^{i,N}_{s-})dX^{i,N}_{s}.
\end{align}

In fact, we can give an explicit construction of this particle system, which will be useful later and we explain now. Let us consider $(U^{i,k})_{i,k \in \N^*}$ a sequence
of independent uniform variables on $[0,1]$ and $(E^{i,k})_{i,k \in \N^*}$ a
sequence of independent exponential random variables with parameter
$\frac{\overline{\lambda}\overline{f}}{\underline{f}}$. These variables are
independent, and independent of the previously defined Brownian motions. We
define the times
$$T^{i,k}=\sum_{l=1}^{k} E^{i,l},$$
and we can order $(T^{i,k},i=1,\dots,N,k\ge 1)$, such that
$0<T^{i_1,k_1}<\dots<T^{i_l,k_l}<\dots$ almost surely. 

Up to the first jump
of~$X^{i,N}$, $Y^{i,N}_t$ is defined as the unique strong solution of
$$ Y^{i,N}_t=Y^i_0+\int_0^t b(s,X^i_0,Y^{i,N}_s)ds+ \int_0^t
\sigma(s,X^i_0,Y^{i,N}_s)dW^i_s. $$
At time $\tau=T^{i_l,k_l}$, the process $X^{i_l,N}$ makes a jump of size~$1$ if
$$U^{i_l,k_l}\le\frac{\underline{f}}{\overline{\lambda}\overline{f}} \frac{\lambda(\tau-,X^{i_l,N}_{\tau-})f(Y^{i_l,N}_{\tau-}) \sum_{j=1}^N \ind{X^{j,N}_{\tau-}=X^{i_l,N}_{\tau-}} }{\sum_{j=1}^N
  f(Y^{j,N}_{\tau-})\ind{X^{j,N}_{\tau-}=X^{i_l,N}_{\tau-}}},$$ and does not jump
otherwise. If a jump occurs, we set
$Y^{i_l,N}_\tau=Y^{i_l,N}_{\tau-}+\gamma(\tau-,X^{i_l,N}_{\tau-},Y^{i_l,N}_{\tau-})$ and, up
to the next jump of $X^{i_l,N}$, we define $Y^{i_l,N}_t$ as the unique strong
solution of
$$ Y^{i_l,N}_t=Y^{i_l,N}_{\tau}+\int_{\tau}^t b(s,X^{i_l,N}_{\tau},Y^{i_l,N}_s)ds+ \int_{\tau}^t
\sigma(s,X^{i_l,N}_{\tau},Y^{i_l,N}_s)dW^i_s, \ t \ge \tau.$$

\subsection{Existence of a solution to~\eqref{pb_mg}}\label{section_existence}
We follow the analysis carried out by~\cite{Meleard}, pages 69 and 70. We
denote by $\mu^{N}=\frac{1}{N} \sum_{i=1}^N
\delta_{(X^{i,N}_t,Y^{i,N}_t)_{t\in[0,T]}} $ the empirical measure given by
the particle system. It is a
random variable taking values in $\cP(E)$. We denote by~$\pi^N \in
\cP(\cP(E))$ the probability law of
$\mu^N$.
 For $\pi \in \cP(\cP(E))$, we denote by
 $$I(\pi)=\int_{\cP(E)}\mu \pi(d\mu) \in \cP(E),$$
the mean of $\pi$.
Let $F:E \rightarrow \R$ be a bounded function which is continuous with respect
to the Skorokhod topology. It induces an application $\cP(E) \rightarrow \R$ --- still denoted by~$F$ with an abuse of notation --- such that
$$F(\mu)=\int_E F(z) \mu(dz).$$ 
Since  $\pi^N$ is by definition the probability law of $\mu^N$,  we have by
using Fubini's Theorem
$\E(F(\mu^N))=\int_{\cP(E)}F(\mu) \pi^N(d\mu)=\int_E F(z) I(\pi^N)(dz).$ On
the other hand, we have $F(\mu^N)=\frac{1}{N}\sum_{i=1}^N
F((X^{i,N}_t,Y^{i,N}_t)_{t\in[0,T]})$. By symmetry,
$(X^{i,N}_t,Y^{i,N}_t)_{t\in[0,T]}$ has the same law as
$(X^{1,N}_t,Y^{1,N}_t)_{t\in[0,T]}$ and we get that:
\begin{equation}\label{eq_pi_N}
  \E[F((X^{1,N}_t,Y^{1,N}_t)_{t\in[0,T]})]=\int_E F(z) I(\pi^N)(dz).
\end{equation}

\begin{lemma}\label{pi_tendue}
The sequence $(\pi^N)_N$ is tight.
\end{lemma}
The proof of Lemma~\ref{pi_tendue} is postponed to Appendix~\ref{App_tension}. Now, 
wWe can consider a subsequence $\pi^{N_k}$ which converges weakly to
$\pi^\infty \in \cP(\cP(E))$. 
Let $q\in \N^*$, $0\le s_1\le\dots\le s_q\le s \le
t$, $\phi\in \cC^{0,2}_b(\R^2,\R)$ and $g_1,\dots,g_q \in\cC_b(\R^2,\R)$ be
bounded functions with bounded derivatives for~$\phi$. We set for $\Q \in
\cP(E)$,
\begin{align}\label{def_F}
  F(\Q)= &\E_\Q \Big[ \Big(\phi(X_t,Y_t)-\phi(X_s,Y_s) - \int_s^t
  \Big\{\frac{\lambda(u,X_u)f(Y_u)}{\E_\Q[f(Y_u)|X_{u}]}
  (\phi(X_u+1,Y_u+\gamma(u,X_u,Y_u)) -\phi(X_u,Y_u))\nonumber \\
 & + b(u,X_u,Y_u)\partial_y\phi(X_u,Y_u)+ 
  \frac{1}{2}\sigma^2(u,X_u,Y_u)\partial^2_y\phi(X_u,Y_u)\Big\}du
  \Big)\prod_{l=1}^q g_l(X_{s_q},Y_{s_q}) \Big]
\end{align}
We have to check that $\Q \mapsto F(\Q)$ is continuous with respect to~$\Q$
for the weak convergence. Since $f$ is continuous by
Assumption~\eqref{assump_f}, we first notice that when $\Q_n$ converges
weakly to~$\Q$, $\E_{\Q_n}[f(Y_t)|X_t=x]$ converges to
$\E_{\Q}[f(Y_t)|X_t=x]$ when $\Q(X_t=x)>0$, unless for an at most countable set of
times~$t$ depending on~$\Q$. Then, following~\cite{Meleard}, it comes out that if
$s,t,s_1,\dots,s_q$ are taken outside a countable set depending on~$\pi^\infty$, $F$
is~$\pi^\infty$-a.s. continuous. In this case, we have:
$$\E[F(\mu^{N_k})^2]\underset{k\rightarrow + \infty}{\rightarrow}
\int_{\cP(E)} F(\Q)^2 \pi^{\infty}(d\Q).$$

By definition of~$\mu^N$, we have:
\begin{eqnarray*}
  F(\mu^N)&=& \frac{1}{N} \sum_{i=1}^N \Big[(M^{i,N}_t-M^{i,N}_s)\prod_{l=1}^q g_l(X^{i,N}_{s_q},Y^{i,N}_{s_q})\Big], \text{ where }  \\
 M^{i,N}_t&=&  \phi(X^{i,N}_t,Y^{i,N}_t)
  -\displaystyle
\int_0^t\Big\{  b(u,X^{i,N}_u,Y^{i,N}_u)\partial_y\phi(X^{i,N}_u,Y^{i,N}_u)+\frac{1}{2}\sigma^2(u,X^{i,N}_u,Y^{i,N}_u)\partial^2_y\phi(X^{i,N}_u,Y^{i,N}_u) 
  \\  && \hspace{-2cm}
+ \frac{\lambda(u,X^{i,N}_{u})f(Y^{i,N}_u) \sum_{j=1}^{N} \ind{X^{j,N}_u=X^{i,N}_{u}} }{\sum_{j=1}^{N}
  f(Y^{j,N}_u)\ind{X^{j,N}_u=X^{i,N}_{u}}}
[\phi(X^{i,N}_u+1,Y^{i,N}_u+\gamma(u,X^{i,N}_u,Y^{i,N}_u))-\phi(X^{i,N}_u,Y^{i,N}_u)]\Big\}du.
\end{eqnarray*}
We observe now that $[M^{i,N},M^{j,N}]_t=0$ for $i\not = j$ since, by construction these
martingales do not jump together and $\langle W^i,W^j
\rangle_t=0$. Therefore, we get that:
$$\E[F(\mu^N)^2]=\frac{1}{N}\E \Big[(M^{1,N}_t-M^{1,N}_s)^2\prod_{l=1}^q
g_l(X^{1,N}_{s_q},Y^{1,N}_{s_q})^2 \Big] \le C/N,$$
thanks to the boundedness assumption made on functions $g_l$ and $\phi$. It
comes out that $F(\Q)=0$, $\pi^\infty(d\Q)$ almost surely. This holds in fact for
any function~$F$ given by~\eqref{def_F}, provided that $s,t,s_1,\dots,s_q$ are
taken outside a countable set depending on~$\pi^\infty$. Since the process
$(X_t,Y_t)$ is càdlàg, this is sufficient to show that any measure in the
support of~$\pi^\infty$ solves the Martingale Problem~\eqref{pb_mg}. In
particular, we get the following result.

\begin{prop}\label{existence_mes_inv}
  There exists a measure~$\probinv \in \cP(E)$ solving the Martingale Problem~\eqref{pb_mg}.
\end{prop}

\subsection{Uniqueness of a solution to~\eqref{pb_mg}}\label{section_unicite}

Let $\probinv\in \mathcal{P}(E)$ denote a probability measure solving the Martingale Problem~\eqref{pb_mg}. We know that such a probability exists thanks to Proposition~\ref{existence_mes_inv}. We want to show that it is indeed
unique. To do so, we consider another probability~$\Q \in
\mathcal{P}(E)$ and study the total variation distance~$V_T(\Psi(\probinv)-\Psi(\Q))$ between~$\probinv$ and
$\Q$ over~$E$. For $t\in [0,T]$, $\probinv,\Q \in \cP(E)$, we denote by
$\probinv\big|_{[0,t]}$ and  $\Q\big|_{[0,t]}$ their restriction to the paths on the time interval~$[0,t]$. We also
set $V_t(\probinv-\Q)$ the total variation distance between  $\probinv\big|_{[0,t]}$ and
$\Q\big|_{[0,t]}$. 

\begin{lemma}\label{lem_tau}
Let $\tau:=\inf \{ t \ge 0,  X^{\probinv}_t \not = X^{\Q}_t \}$ denote the first
time when $X^{\probinv}$ and $X^{\Q}$ do not jump together. We have
$$V_T(\Psi(\probinv)-\Psi(\Q)) \le 2 \P(\tau \le T). $$
\end{lemma}
\begin{proof}
  Let us recall that for any signed measure~$\eta$ on~$E$, the
total variation of~$\eta$ is given by 
$$V(\eta)=\eta^+(E)+\eta^-(E),$$
where $\eta=\eta^+-\eta^-$ is the Hahn-Jordan decomposition
of~$\eta$. Besides, we clearly have
$$\frac{1}{2} V(\eta) \le \tilde{V}(\eta) \le V(\eta),$$
where $\tilde{V}(\eta)= \sup \{ |\eta(A)|, A\subset E \ measurable\}$.

We have for any measurable set~$A$ of~$E$,
\begin{eqnarray*}
&&  \P((X^{\probinv}_t,Y^{\probinv}_t)_{t\in[0,T]} \not =
(X^{\Q}_t,Y^{\Q}_t)_{t\in[0,T]}) \\&&\ge \P((X^{\probinv}_t,Y^{\probinv}_t)_{t\in[0,T]} \in A
, (X^{\Q}_t,Y^{\Q}_t)_{t\in[0,T]} \not \in A)\\
&&=\P((X^{\probinv}_t,Y^{\probinv}_t)_{t\in[0,T]} \in A)-\P((X^{\probinv}_t,Y^{\probinv}_t)_{t\in[0,T]} \in A
, (X^{\Q}_t,Y^{\Q}_t)_{t\in[0,T]}  \in A) \\&&\ge 
\P((X^{\probinv}_t,Y^{\probinv}_t)_{t\in[0,T]} \in A)-\P((X^{\Q}_t,Y^{\Q}_t)_{t\in[0,T]}
\in A).
\end{eqnarray*}
By taking the supremum over~$A$, we get
$$\P((X^{\probinv}_t,Y^{\probinv}_t)_{t\in[0,T]} \not =
(X^{\Q}_t,Y^{\Q}_t)_{t\in[0,T]}) \ge \tilde{V}_T(\Psi(\probinv)-\Psi(\Q)),$$
which gives the claim.\end{proof}

Up to time~$\tau$, we
have $(X^{\probinv}_t,Y^{\probinv}_t)=(X^{\Q}_t,Y^{\Q}_t)$, and we get that
$$  \ind{\tau \le t} - \int_0^t \ind{\tau > s}\lambda(s,X^{\probinv}_s)f(Y^{\probinv}_s)
\left|\frac{1}{\varphi^{\probinv}(s,X^{\probinv}_s)}-\frac{1}{\varphi^{\Q}(s,X^{\probinv}_s)}
\right|ds$$
is a martingale. In particular, we have
\begin{eqnarray*}
  \frac{d \P(\tau \le t)}{dt}&=&\E\left[\ind{\tau > t}\lambda(t,X^{\probinv}_t)f(Y^{\probinv}_t)
\left|\frac{1}{\varphi^{\probinv}(t,X^{\probinv}_t)}-\frac{1}{\varphi^{\Q}(t,X^{\probinv}_t)}
\right| \right] \\ &\le& \frac{\overline{\lambda} \overline{f} }{\underline{f}^2}
\E\left[|\varphi^{\Q}(t,X^{\probinv}_t)-\varphi^{\probinv}(t,X^{\probinv}_t) |\right].
\end{eqnarray*}
Now, let us observe that
$\varphi^{\probinv}(t,x)-\varphi^{\Q}(t,x)=\frac{\E_\probinv[f(Y_t)\ind{X_t=x}]-\E_{\Q}[f(Y_t)\ind{X_t=x}]}{\probinv(X_t=x)}
+\frac{\E_{\Q}[f(Y_t)|X_t=x]}{\probinv(X_t=x)}[\Q(X_t=x)-\probinv(X_t=x)]$. Thus, we get
that $|\varphi^{\probinv}(t,x)-\varphi^{\Q}(t,x)|\le  \frac{1}{\probinv(X_t=x)} \left(|
  \E_\probinv[f(Y_t)\ind{X_t=x}]-\E_{\Q}[f(Y_t)\ind{X_t=x}]|+\overline{f} |\Q(X_t=x)-\probinv(X_t=x)|
\right), $ and therefore (we use here that $\probinv$ is invariant, and is the law
of $(X^\probinv_t,Y^\probinv_t)_{t \ge 0}$)
\begin{eqnarray*}
  \E[|\varphi^{\Q}(t,X^{\probinv}_t)-\varphi^{\probinv}(t,X^{\probinv}_t) |] &\le&  \sum_{x \in
    \cL_M} |
  \E_\probinv[f(Y_t)\ind{X_t=x}]-\E_{\Q}[f(Y_t)\ind{X_t=x}]|+\overline{f} |\Q(X_t=x)-\probinv(X_t=x)| \\
  & \le & 2\overline{f}  V_t(\Psi(\probinv)-\Psi(\Q)). 
\end{eqnarray*}
To check the last inequality, one has to observe that for a simple nonnegative
function $f(y)=\sum_{i=1}^n f_i \ind{A_i}$, where $A_i \subset \R$ are Borel sets, we
have
$|\E_\probinv[f(Y_t)\ind{X_t=x}]-\E_{\Q}[f(Y_t)\ind{X_t=x}]|\le \sum_{i=1}^n f_i
|\probinv(X_t=x, Y_t \in A_i)-\Q(X_t=x, Y_t \in A_i)|\le\overline{f}
V_t(\probinv-\Q) $. By passing to the limit, this property holds for
any bounded measurable (hence continuous) function~$f$.

To sum up, we have for $t \in [0,T]$,
$$V_T(\Psi(\probinv)-\Psi(\Q))\le 2 \P(\tau \le T) \le
4\frac{\overline{\lambda} \ \overline{f}^2
}{\underline{f}^2}\int_0^TV_s(\probinv-\Q)ds \le 4\frac{\overline{\lambda} \
  \overline{f}^2 }{\underline{f}^2} TV_T(\probinv-\Q),$$
since $V_0(\Psi(\probinv)-\Psi(\Q))=0$ and $s\mapsto V_s(\probinv-\Q)$ is nondecreasing.
Let us recall that the set of bounded countably additive measures on~$E$ endowed with the
total variation norm is a Banach space. If $ 4\frac{\overline{\lambda} \
  \overline{f}^2 }{\underline{f}^2} T<1$, we get by
the Banach fixed point theorem that $\Psi$ has a unique fixed
point that is necessarily~$\probinv$. Otherwise, we get by iterating that:
$$V_T(\Psi^{(k)}(\probinv)-\Psi^{(k)}(\Q))\le 4\frac{\overline{\lambda} \
  \overline{f}^2 }{\underline{f}^2} \int_0^T
V_s(\Psi^{(k-1)}(\probinv)-\Psi^{(k-1)}(\Q))ds\le
\frac{\left[4\frac{\overline{\lambda} \
  \overline{f}^2 }{\underline{f}^2} T \right]^k}{k!}V_T(\probinv-\Q). $$
When $k$ is large enough, $\frac{\left[4\frac{\overline{\lambda} \
  \overline{f}^2 }{\underline{f}^2} T \right]^k}{k!}<1$ and
$\Psi^{(k)}$ is a contraction mapping. Thus, $\Psi^{(k)}$ has a unique fixed
point that is necessarily~$\probinv$.

This concludes the proof of Theorem~\ref{thm_continuous}. Besides, using the notations of Section~\ref{section_existence}, we get that any convergent subsequence of~$\pi^N$ should converge to~$\delta_{\probinv}(d\Q)$. This gives the weak convergence of~$\pi^N$ towards~$\delta_{\probinv}(d\Q)$. By~\eqref{eq_pi_N}, we get that any particle converges in law towards~$\probinv$.

\subsection{Convergence speed towards $\probinv$}\label{section_speed}

Now that we have proved that each particle converges to the invariant probability measure, we are interested in characterizing the speed of convergence of the interacting particle system towards this measure. This question is of practical importance, since one would like to use the following approximation
\begin{equation}\label{approx_ips}
\E_{\probinv}[F((X_t,Y_t)_{t\in [0,T]})]\approx \frac{1}{N} \sum_{i=1}^N
F(({X}^{i,N}_t,{Y}^{i,N}_t)_{t\in [0,T]}),
\end{equation}
and have an estimate of the error involved.

First, we need to introduce some additional notations. We consider the same particle system
$(X^{i,N}_t,Y^{i,N}_t)_{t\ge 0}$ as
in Section~\ref{sec_ips}, constructed with the random variables $T^{i,k}$, $U^{i,k}$ and
$(W^i_t,t\ge 0)$. With these variables, we construct now the processes
$(\bar{X}^i_t, \bar{Y}^i_t)_{t\ge 0}$ as the unique solution of
\begin{equation}\label{continuous_std_sde}
\begin{cases}
  \bar{X}^i_t=\bar{X}^i_0+\sum_{k, T^{i,k}\le t} \mathbf{1}_{\left\{U^{i,k} \le\frac{\underline{f}}{\overline{\lambda}\overline{f}} \frac{f(\bar{Y}^i_{T^{i,k}-})\lambda(T^{i,k}-,\bar{X}^i_{T^{i,k}-})}{\varphi^{\probinv}(T^{i,k}-,\bar{X}^i_{T^{i,k}-})}\right\}}\\
  \bar{Y}^i_t=\bar{Y}^i_0+\int_0^t b(s,\bar{X}^i_s,\bar{Y}^i_s)ds +\int_0^t \sigma (s,\bar{X}^i_s,\bar{Y}^i_s)dW^i_s + \int_0^t \gamma(s-,\bar{X}^i_{s-},\bar{Y}^i_{s-})d\bar{X}^i_s.
\end{cases}
\end{equation}


By construction, the law of $(\bar{X}^i_t, \bar{Y}^i_t)_{t\in[0,T]}$ is the
invariant probability law~$\probinv$, since $\Psi(\probinv)=\probinv$. By using the same argument as in Lemma~\ref{lem_tau}, we have:
$$V_T(\cL((X^{1,N}_t,Y^{1,N}_t)_{t\in [0,T]})-\probinv)\le 2
\P((X^{1,N}_t,Y^{1,N}_t)_{t\in [0,T]} \not
=(\bar{X}^{1}_t,\bar{Y}^{1}_t)_{t\in [0,T]}) = 2 \P(\tau^1\le T),$$
where $\tau^1=\inf \{ t\ge 0, \bar{X}^{1}_t \not = X^{1,N}_t \}$. We also set
for $i=2,\dots,N$, $\tau^i=\inf \{ t\ge 0, \bar{X}^{i}_t \not = X^{i,N}_t \}$.
\begin{prop}\label{prop_speed}
Let us assume that $\cL_{init}$ is such that $\P(X_0=x)>0$ for any $x \in \cL_M$. Then, there is a constant~$K>0$ such that 
$$ \P(\tau^1\le T) \le \frac{K}{\sqrt{N}}. $$
\end{prop}

\begin{proof}
By construction, the processes  $\bar{X}^{1}$ and $X^{1,N}$ may become different at the times $T^{1,k}$ if $U^{1,k}\le \frac{\underline{f}}{\overline{\lambda}\overline{f}}   \frac{\lambda({T^{1,k}-},\bar{X}^{1}_{T^{1,k}-})
      f(\bar{Y}^{1}_{T^{1,k}-})}{\varphi^\probinv({T^{1,k}-},\bar{X}^{1}_{T^{1,k}-})}$ and $U^{1,k}>\frac{\underline{f}}{\overline{\lambda}\overline{f}} \frac{\lambda({T^{1,k}-},X^{1,N}_{T^{1,k}-})
      f(Y^{1,N}_{T^{1,k}-})}{\frac{\sum_{j=1}^N f(Y^{j,N}_{T^{1,k}-}) \ind{X^{j,N}_{T^{1,k}-}=X^{1,N}_{T^{1,k}-}}
      }{\sum_{j=1}^N \ind{X^{j,N}_{T^{1,k}-}=X^{1,N}_{T^{1,k}-}}}}$, or
    conversely. Thus, $\mathbf{1}_{\{\tau^1\le t\}}-\int_0^t \ind{\tau^1 > s}\frac{\underline{f}}{\overline{\lambda}\overline{f}} 
  \left|\frac{\lambda(s,\bar{X}^{1}_s)
      f(\bar{Y}^{1}_s)}{\varphi^\probinv(s,\bar{X}^{1}_s)}-\frac{\lambda(s,X^{1,N}_s)
      f(Y^{1,N}_s)}{\frac{\sum_{j=1}^N f(Y^{j,N}_s) \ind{X^{j,N}_s=X^{1,N}_s}
      }{\sum_{j=1}^N \ind{X^{j,N}_s=X^{1,N}_s}}} \right|ds$ is a martingale and we have
\begin{eqnarray*}
\P(\tau^1\le T) &=& \frac{\underline{f}}{\overline{\lambda}\overline{f}} \E \left[ \int_0^T \ind{\tau^1 > t}
  \left|\frac{\lambda(t,\bar{X}^{1}_t)
      f(\bar{Y}^{1}_t)}{\varphi^\probinv(t,\bar{X}^{1}_t)}-\frac{\lambda(t,X^{1,N}_t)
      f(Y^{1,N}_t)}{\frac{\sum_{j=1}^N f(Y^{j,N}_t) \ind{X^{j,N}_t=X^{1,N}_t}
      }{\sum_{j=1}^N \ind{X^{j,N}_t=X^{1,N}_t}}} \right|dt\right].
\end{eqnarray*}
Let us observe that on $\{ \tau^1 > t \}$, we have
$(\bar{X}^{1}_t,\bar{Y}^{1}_t)=(X^{1,N}_t,Y^{1,N}_t)$. Therefore,
\begin{eqnarray}\label{vitesse_1}
\P(\tau^1\le T) &\le & \underline{f}\E \left[ \int_0^T \ind{\tau^1 > t}
  \left|\frac{1}{\varphi^\probinv(t,\bar{X}^{1}_t)}-\frac{1+\sum_{j=2}^N \ind{X^{j,N}_t=\bar{X}^{1}_t}}{ f(Y^{1,N}_t)+\sum_{j=2}^N
        f(Y^{j,N}_t) \ind{X^{j,N}_t=\bar{X}^{1}_t}        
      } \right|dt\right]. \ \ \ \ \ \ \ \
\end{eqnarray}
Now, we study $ \left|\frac{1}{\varphi^\probinv(t,x)}-\frac{1+\sum_{j=2}^N \ind{X^{j,N}_t=x}}{ f(Y^{1,N}_t)+\sum_{j=2}^N f(Y^{j,N}_t) \ind{X^{j,N}_t=x} } \right|$ for $x \in \cL_M$, and set $$\bar{q}_t(x)=\probinv(X_t=x).$$
      
When $\bar{q}_t(x)>0$, we have 
\begin{eqnarray*}
  &&\hspace{-1.5cm}\left|\frac{1}{\varphi^\probinv(t,x)}-\frac{1+\sum_{j=2}^N \ind{X^{j,N}_t=x}}{ f(Y^{1,N}_t)+\sum_{j=2}^N
        f(Y^{j,N}_t) \ind{X^{j,N}_t=x}        
      } \right| 
    \le \left| \frac{ \bar{q}_t(x)-\frac{1}{N} \left(1+\sum_{j=2}^N
    \ind{X^{j,N}_t=x}\right) }{\E_{\probinv}[f(Y_t) \ind{X_t=x}]}\right| 
\\
&&+\frac{1}{\underline{f} \E_{\probinv}[f(Y_t) \ind{X_t=x}]  } \left|\E_{\probinv}[f(Y_t) \ind{X_t=x}] - \frac{1}{N}\left(f(Y^{1,N}_t)+ \sum_{j=2}^N
    f(Y^{j,N}_t)\ind{X^{j,N}_t=x}\right) \right|
  \\ &&\le \frac{1}{\underline{f} \bar{q}_t(x)} \left|\bar{q}_t(x)-\frac{1}{N} \left(1+\sum_{j=2}^N
    \ind{X^{j,N}_t=x}\right) \right|  \\ &&+ \frac{1}{\underline{f}^2\bar{q}_t(x)}\left|\E_{\probinv}[f(Y_t) \ind{X_t=x}] - \frac{1}{N}\left(f(Y^{1,N}_t)+ \sum_{j=2}^N
    f(Y^{j,N}_t)\ind{X^{j,N}_t=x}\right) \right|.
  \end{eqnarray*}
We analyze these two terms in a similar manner. We introduce
$(\bar{X}^{N+1}_t,\bar{Y}^{N+1}_t)_{t\in[0,T]}$ another copy of
$(\bar{X}^1_t,\bar{Y}^1_t)_{t\in[0,T]}$, which is independent from all other
existing processes. We have:
\begin{eqnarray}\label{vitesse_2}
\left| \bar{q}_t(x)-\frac{1}{N} \left(1+\sum_{j=2}^N
    \ind{X^{j,N}_t=x}\right) \right|  
    &\le& \left| \bar{q}_t(x)-\frac{1}{N} \sum_{j=2}^{N+1}
    \ind{\bar{X}^{j}_t=x} \right| + \frac{1}{N}\left|\sum_{j=2}^N
    \ind{\bar{X}^{j}_t=x}-\ind{X^{j,N}_t=x} \right| + \frac{1}{N} \nonumber\\
  &\le &\left| \bar{q}_t(x)-\frac{1}{N} \sum_{j=2}^{N+1}
    \ind{\bar{X}^{j}_t=x} \right| + \frac{1}{N} \sum_{j=2}^N
    \ind{\tau^j\le t} + \frac{1}{N},
  \end{eqnarray}
  since $\bar{X}^{j}_t$ and $X^{j,N}_t$ may be different only on $\{\tau^j \le
  t\}$. Similarly, we have
  
\begin{eqnarray}\label{vitesse_3}
&&\left|\E_{\probinv}[f(Y_t) \ind{X_t=x}] - \frac{1}{N}\left(f(Y^{1,N}_t)+ \sum_{j=2}^N
    f(Y^{j,N}_t)\ind{X^{j,N}_t=x}\right) \right| \nonumber\\
  &&\le \left|\E_{\probinv}[f(Y_t) \ind{X_t=x}] -\frac{1}{N} \sum_{j=2}^{N+1}
    f(\bar{Y}^{j}_t)\ind{\bar{X}^{j}_t=x} \right|+\frac{\overline{f}}{N} \sum_{j=2}^N
    \ind{\tau^j\le t} + \frac{\overline{f}}{N}.
\end{eqnarray}
We introduce
$$A(x)=\frac{1}{ \bar{q}_t(x)}\left( \left| \bar{q}_t(x)-\frac{1}{N} \sum_{j=2}^{N+1}
    \ind{\bar{X}^{j}_t=x}   \right| + \frac{1}{\underline{f}}\left|\E_{\probinv}[f(Y_t) \ind{X_t=x}] -\frac{1}{N} \sum_{j=2}^{N+1}
    f(\bar{Y}^{j}_t)\ind{\bar{X}^{j}_t=x} \right|\right).$$ This is a random function
  which is independent from $(\bar{X}^1_t,\bar{Y}^1_t)_{t\ge 0}$.
From~\eqref{vitesse_1},~\eqref{vitesse_2} and~\eqref{vitesse_3}, we obtain by
observing that $\tau^j$ and $\tau^1$ have the same law under~$\P$:
\begin{equation}\label{vitesse_4}
 \P(\tau^1\le T) \le    \int_0^T
  \E[A(\bar{X}^1_t)] + \left( 1+\frac{\overline{f}}
    {\underline{f}}\right) \left(\frac{1}{N} \E \left[
    \frac{1}{\bar{q}_t(\bar{X}^1_t)} \right] + \E \left[
    \frac{1}{\bar{q}_t(\bar{X}^1_t)} \mathbf{1}_{\tau^2\le t} \right] \right)    dt.
\end{equation}
First, we observe that $\E[
    \frac{1}{\bar{q}_t(\bar{X}^1_t)}]=\sum_{x, \ s. t.\ \bar{q}_t(x)>0} \frac{\bar{q}_t(x)}{\bar{q}_t(x)} \le M+1$.
Thanks to the independence of~$A(x)$ and $\bar{X}^1_t$,
$\E[A(\bar{X}^1_t)|\bar{X}^1_t=x]=\E[A(x)]$. On the one hand, we have:\\
$\E\left[ \left| \bar{q}_t(x)-\frac{1}{N} \left(\sum_{j=2}^{N+1}
    \ind{\bar{X}^{j}_t=x}\right) \right| \right] \le  \sqrt{\E\left[ \left( \bar{q}_t(x)-\frac{1}{N} (\sum_{j=2}^{N+1}
    \ind{\bar{X}^{j}_t=x}) \right)^2
\right]}=\frac{\sqrt{\bar{q}_t(x) (1-\bar{q}_t(x)) }}{\sqrt{N}}$. On the
other hand, we have similarly that:\\
\begin{align*}
&\E\left[\left|\E_{\probinv}[f(Y_t) \ind{X_t=x}] - \frac{1}{N}\left( \sum_{j=2}^{N+1}
    f(\bar{Y}^{j}_t)\ind{\bar{X}^{j}_t=x}\right) \right|\right] \\& \le
\sqrt{\E \left[ \left(\E_{\probinv}[f(Y_t) \ind{X_t=x}]-
      \frac{1}{N} \sum_{j=2}^{N+1}
    f(\bar{Y}^{j}_t)\ind{\bar{X}^{j}_t=x})  \right)^2\right]} \le\sqrt{\frac{1}{N} \E_{\probinv}(\ind{X_t=x}f^2(Y_t)) }\le\sqrt{\frac{\overline{f}^2}{N} \bar{q}_t(x) }
.
\end{align*}
Finally, we obtain that:
$$\E[A(x)] \le
\left(1+\frac{\overline{f}}{\underline{f}} \right)
\frac{1}{\sqrt{\bar{q}_t(x)}\sqrt{N}}.$$
By using the tower property of the conditional expectation, we get:
$$\E[A(\bar{X}^1_t)] \le
\left(1+\frac{\overline{f}}{\underline{f}} \right)
\frac{\sqrt{M+1}}{\sqrt{N}},$$
since we have $\sum_{x\in \cL_M}\sqrt{\bar{q}_t(x)} \le \sqrt{M+1}$ by the
Cauchy-Schwarz inequality.
From~\eqref{vitesse_4}
\begin{equation}\label{vitesse_5}
  \P(\tau^1\le T) \le  \left(1+\frac{\overline{f}}{\underline{f}} \right) \left(
   \frac{\sqrt{M+1}}{\sqrt{N}} +\frac{M+1}{N} +\int_0^T \E\left[
    \frac{1}{  \bar{q}_t(\bar{X}^1_t)} \mathbf{1}_{\tau^2\le t} \right]    dt
  \right).
\end{equation}
So far, we have not used the assumption $\bar{q}_0(x)=\P(X_0=x)>0$ for any $x\in \cL_M$. Since the jump intensity is bounded by $\frac{\overline{\lambda}\overline{f}}{\underline{f}}$, we necessarily have $ \bar{q}_t(x)\ge e^{-\frac{\overline{\lambda}\overline{f}}{\underline{f}} T}\bar{q}_0(x),$ for $t\in [0,T]$. Thus, there exists a constant $C>0$ (depending on $T$ and $\cL_{init}$) such that $\frac{1}{  \bar{q}_t(x)}\le C$ for $t\in[0,T]$, $x\in \cL_M$. Since $\tau^2$ and $\tau^1$ have the same law under~$\P$, this gives
\begin{equation}\label{vitesse_6}
  \P(\tau^1\le T) \le  \left(1+\frac{\overline{f}}{\underline{f}} \right) \left(
   \frac{\sqrt{M+1}}{\sqrt{N}} +\frac{M+1}{N} +C \int_0^T   \P(\tau^1\le t)  dt
  \right),
\end{equation}
and we easily conclude by Gronwall's lemma. 
\end{proof}

Now, we can have an estimate of the accuracy given by the approximation~\eqref{approx_ips}.
We assume that $F:E\rightarrow \R$ is a bounded measurable function. Then, we know by the central limit theorem that 
$$\sqrt{N}\left( \frac{1}{N} \sum_{i=1}^N
F((\bar{X}^{i}_t,\bar{Y}^{i}_t)_{t\in [0,T]})-\E_\probinv[F((X_t,Y_t)_{t\in [0,T]})]
\right) \underset{law}{\rightarrow} \cN(0,\sigma^2), $$
where $\sigma^2$ is the variance of $F((X_t,Y_t)_{t\in [0,T]})$ under~$\probinv$.

Since $F$ is bounded by a constant~$K>0$, we have by Proposition~\ref{prop_speed}
$$N^\alpha \left|\frac{1}{N} \sum_{i=1}^N
F((\bar{X}^{i}_t,\bar{Y}^{i}_t)_{t\in [0,T]})-\frac{1}{N} \sum_{i=1}^N
F(({X}^{i,N}_t,{Y}^{i,N}_t)_{t\in [0,T]})\right| \le \frac{K}{N^{1-\alpha}} \sum_{i=1}^N
\ind{\tau^i \le T}, $$
which converges for the $L^1$-norm to~$0$ when $N\rightarrow+\infty$  for $\alpha<1/2$. Combining both results, we finally get a lower estimate of the convergence rate. 
\begin{corollary}
Under the assumptions of Proposition~\ref{prop_speed}, we have for any bounded measurable function $F:E\rightarrow \R$ and any $0<\alpha<1/2$,
$$N^\alpha \left|\frac{1}{N} \sum_{i=1}^N
F(({X}^{i,N}_t,{Y}^{i,N}_t)_{t\in [0,T]})-\E_\probinv[F((X_t,Y_t)_{t\in [0,T]})]\right|\rightarrow 0 \text{ in probability}. $$
\end{corollary}

\begin{remark}
  To prove Proposition~\ref{prop_speed}, we have assumed that $\P(X_0=x)>0$ for any $x\in \cL_M$. In fact, the same proof would work if we assumed that there existed $x_0 \in \cL_M$ such that $\P(X_0\ge x_0)=1$ and $\P(X_0=x)>0$ for any $x\ge x_0$.

 However, in practice, it would have been nice to treat the case $\P(X_0=x_0)=1$ for some $x_0\in \cL_M$, since we know at the beginning how many firms have already defaulted. Heuristically from~\eqref{vitesse_5}, we may hope to have for large $N$ that $ \E\left[
    \frac{1}{  \bar{q}_t(\bar{X}^1_t)} \mathbf{1}_{\tau^2\le t} \right] \le C \P(\tau^1\le t) $ since $\bar{X}^1$ and $\tau^2$ are asymptotically independent, $ \E\left[
    \frac{1}{  \bar{q}_t(\bar{X}^1_t)} \right] \le M+1$, and $\tau^2$ and $\tau^1$ have the same law. This would be enough to conclude. Unfortunately, despite our investigations, we have not been able to prove this formally. However, we still observe a convergence speed of $1/\sqrt{N}$ when $X_0=0$ on our numerical experiments (Section~\ref{sec_num}).
\end{remark}

\section{Numerical results}\label{sec_num}

In this section, we illustrate the theoretical results obtained in the previous
sections. Let us recall that the local intensity (LI) model is a Markov chain with unit jumps occurring with the rate~$\lambda(t-,x)$ and that the SLI model is given by equation~\eqref{continuous}.
First, we highlight that the LI and SLI models have the
same marginal distributions but different laws as processes. Second, we study
the convergence of the interacting particle system and obtain numerical
simulations showing a central limit theorem.

For our numerical experiments we consider two different models for the process
$Y$ described below and we assume that the local intensity $\lambda$ and the
function $f$ are given by
\begin{align*}
  \lambda(t, x) & = \bar \lambda \left(1 - \frac{x}{M}\right) \\
  f(x) & = (x \vee \underline{f} ) \wedge \overline{f}
\end{align*}
where the number of defaultable entities $M$ will be taken equal to $M=125$
from now on.  This choice of $\lambda(t,x)$ corresponds to $M$ independent default times with intensity~$\frac{\bar \lambda }{M}$. We also assume that there is no default at the beginning, i.e. $X_0=0$.

\begin{enumerate}
  \item In the framework proposed by~\cite{Lopatin}, the process $Y$ is a
    continuous process satisfying a CIR type SDE
    \begin{equation}
      \label{eq:lopatin}
      dY_t = \kappa ( \lambda(t, X_{t-}) - Y_t) + \sigma \sqrt{Y_t} dW_t 
    \end{equation}
    where $\kappa, \sigma >0$. To sample such a process, we use the second-order
    discretization scheme for the CIR diffusion given in~\cite{AA_CIR}.

  \item In the framework of \cite{Arnsdorf}, the process $Y$ is no more
    continuous and may jump when a new default occurs. The process $Y$
    solves the following SDE
    \begin{equation}
      \label{eq:arnsdorf}
      dY_t = - a Y_t \log(Y_t)dt + \sigma Y_t dW_t + \gamma Y_{t^-} dX_t 
    \end{equation}
    where $a, \sigma >0$ and $\gamma \ge 0$. Remember that $X$ only has positive
    jumps. For discretization purposes, note that between two jump times of $X$, 
    $\log(Y)$ solves the following Ornstein Uhlenbeck SDE
    \begin{equation*}
      dZ_t = (- a Z_t + \frac{1}{2} \sigma^2 ) dt + \sigma dW_t.
    \end{equation*}
    Even though we could have sampled exactly in this particular case the Gaussian increments of~$Z$, we discretize $Y$ using the Euler scheme on~$Z$ in our simulation since we would make this choice for more general SDEs on~$Y$.
\end{enumerate}

The LI model can be simulated very easily using a standard Monte--Carlo approach
as it is sufficient to know how to sample a Poisson process with intensity
$\lambda$; we do not need any interacting particle system.

\subsection{Practical implementation}

In this part, we describe our implementation of the particle system to sample
from the distribution of $(X_t, Y_t)_{0 \le t \le T}$. We recall that the process $X$ has no continuous part and makes jumps of size
$1$. The process $Y$ has a continuous part and may jump at the same times as
$X$. We consider a regular time grid of $[0,T]$ with step size $h = \frac{T}{D}$: $s_k=kh$,
$0 \le k\le D$.  Assume we have already discretized $Y$ up
to time $s_k$, the discretization of $Y$ at time $s_{k+1}$ is built in the
following way:
\begin{itemize}
  \item If the process $X$ does not jump between time $s_k$ and time $s_{k+1}$ we
    use the increment of a standard discretization scheme.
  \item If the process $X$ jumps at time $s$ with $s_k < s < s_{k+1}$, we proceed
    in three steps: apply the previous case between times $s_k$ and $s$,
    integrate the jump at time $s$ and finally apply the previous case again
    between time $s$ and $s_{k+1}$.
\end{itemize}
This scheme ensures all the $Y^i$ are at least discretized on the regular grid
$\{ s_0, s_1, \dots,  s_D\}$.

\begin{algorithm*}[h]
  \caption{(The naive particle system algorithm with complexity $O(DN + N^2)$.)}\label{algodetail}
  \begin{algorithmic}[1]
    \State $t=0$ \Comment{Current date.}
    \State $t_i = 0$ for all $i=1\dots N$ \Comment{Last discretization date for
    the particle $i$.}
    \State Sample  $(X^i,Y^i)$ independently according to the initial law. 
    \State $s_k = 0$ \Comment{ Last date on the regular grid: $s_k\le t<s_{k+1}$.}
    \While{$t \le T$}\label{loop_outer}
      \State $t' =t+ {\mathcal E} \left( \frac{\bar \lambda \bar f}{\underline f} N \right)$
      \Comment{$t'$ is the potential next jump in the whole particle system.}
      \While{$t' > s_k + h$} \label{loop_time}
        \State $s_k = s_k+h$
        \For{$i=1$ {\bf to} $N$}
          \State Update the discretization of $Y^i$ from time $t_i$ to time 
          $s_k$.
          \State $t_i = s_k$
        \EndFor
        \State $t=s_k$
      \EndWhile
      \State $I =$ uniform r.v. with values in $\{1, \dots, N\}$. \Comment{Index
      of the particle which may jump.}
      \State Compute the conditional expectation
      \begin{align}
        \label{eq:expect_cond}
        E = \frac{\sum_{l=1}^N \ind{X^{l}=X^{I}} }
        {\sum_{l=1}^N f(Y^{l})\ind{X^{l}=X^{I}}}.
      \end{align}

      \State $R = \frac{\underline f}{\bar \lambda  \bar f}\lambda(t',X^{I})f(Y^{I}) E$ 
        \Comment{Compute the acceptance ratio.}
      \State $U =$ uniform r.v. in $[0,1]$.
      \If{$U < R$}
        \Comment{We accept the jump.}
        \State Discretize $Y^I$ up to time $t'$.
        \State $Y^I=Y^I+\gamma(t',X^I,Y^I)$ 
        \State $X^I = X^I+1$ 
        \State $t_I = t'$
      \EndIf
      \State $t=t'$
    \EndWhile
  \end{algorithmic}
\end{algorithm*}

\paragraph{Computational complexity.} Studying the complexity is of prime
importance when proposing a numerical algorithm.

On the one hand, there are $D$ discretization times at which we recalculate the
values of each $Y^i$, which requires $O(DN)$ operations. On the other hand, the
average total number of proposed jump dates (ie. the number of steps in
loop~\ref{loop_outer}) is given by the expectation of the underlying Poisson
process at time $T$ : $N T \frac{\bar \lambda  \bar f}{\underline f}$.  The
computation cost of the ratio~\eqref{eq:expect_cond}  is $O(N)$ and the
complexity is $O(ND + N^2 T \frac{\bar \lambda  \bar f}{\underline f})$. Hence,
for fixed model parameters, the overall complexity of our approach is bounded
by $O(D N +  N^2)$. In practice, $N$ is much larger than $D$, and the most
computationally demanding part of the algorithm is the numerical approximation
of the condition expectation involved in the jump intensity. 

The complexity of the interacting particle approach can be well improved if
during the algorithm we keep track of the following two quantities involved in
Equation~\eqref{eq:expect_cond}
\begin{align}
  \label{eq:DNk}
  D_t^j  = \sum_{l=1}^N f(Y^{l}_{t-})\ind{X^{l}_{t-}=j} \quad \text{and} \quad
  N_t^j = \sum_{l=1}^N \ind{X^{l}_{t-}=j} \quad j=0, \dots,M.
\end{align}
Since the processes $X^{l}$ are unit-jump increasing these quantities clearly vanish for $j>\max_l X^l_{t-}$. 
\begin{algorithm*}[h!b]
  \caption{(The improved particle system algorithm with complexity $O(DN)$.)}\label{algodetail2}
  \begin{algorithmic}[1]
    \State $t=0$
    \State $t_i = 0$ for all $i=1\dots N$ \Comment{Last discretization date for
    the particle $i$.}
    \State Sample  $(X^i,Y^i)$ independently according to the initial law. 
    \State Set  $D^l=N^l=0$ for $0\le l \le M$.
    \For{$i=1$ {\bf to} $N$} \Comment {Calculate $D$ and $N$. We can directly set  $N^0 = N$, $D^0 = N f(Y_0)$, $D^l=N^l=0$ for $1\le l \le M$ when the initial law is a Dirac mass at $(0,Y_0)$.}
    \State $N^{X^i}=N^{X^i}+1$, $D^{X^i}=D^{X^i}+f(Y^i)$
    \EndFor 
    \State $s_k = 0$ \Comment {Last date on the regular grid.}
    \While{$t \le T$}
      \State $t' =t+ {\mathcal E} \left( \frac{\bar \lambda \bar f}{\underline f} N \right)$
      \Comment{$t'$ is the potential next jump in the whole particle system.}
      \While{$t' > s_k + h$}
        \State $s_k = s_k+h$
        \For{$i=1$ {\bf to} $N$}
          \State Update the discretization of $Y^i$ from time $t_i$ to time 
          $s_k$.
          \State $t_i = s_k$
        \EndFor
        \State $t=s_k$
        \State Reinitialize $D^l=N^l=0$ for $0\le l \le M$.
        \For{$i=1$ {\bf to} $N$} \Comment {Recalculate D and N.}
          \State $N^{X^i}=N^{X^i}+1$, $D^{X^i}=D^{X^i}+f(Y^i)$
        \EndFor
      \EndWhile
      \State $I =$ uniform r.v. with values in $\{1, \dots, N\}$ \Comment{Index
      of the particle which may jump.}     
      \State $R = \frac{\underline f}{\bar \lambda  \bar f}\lambda(t',X^{I})f(Y^{I}) 
      \frac{N^{X^I}}{D^{X^I}}$ 
        \Comment{Compute the acceptance ratio.}
      \State $U =$ uniform r.v. in $[0,1]$
      \If{$U < R$}
        \Comment{We accept the jump.}       
        \State $N^{X^I}=N^{X^I}-1$, $D^{X^I}=D^{X^I}-f(Y^I)$
        \State Discretize $Y^I$ up to time $t'$.
        \State $Y^I=Y^I+\gamma(t',X^I,Y^I)$ 
        \State $X^I = X^I+1$
        \State $N^{X^I}=N^{X^I}+1$, $D^{X^I}=D^{X^I}+f(Y^I)$
        \State $t_I = t'$
      \EndIf
      \State $t=t'$
    \EndWhile
  \end{algorithmic}
\end{algorithm*}

Let $t$ be the last proposed jump time of the particle system. These two
vectors can be easily updated at time $t'$ which denotes the next possible jump
time.  If some ticks of the regular grid lie in $[t, t')$, we set $t$ as the
last discretization date in this interval and recompute vectors $(D_t^j)_j$ and
$(N_t^j)_j$ using Equation~\eqref{eq:DNk}. This happens $D$ times in the
algorithm and can be done with $O(M+N)$ operations as explained in
Algorithm~\ref{algodetail2}.
\begin{description}
  \item[Case 1:] If the proposed jump at time $t'$ is not accepted, there is
    nothing to compute:
    \begin{align*}
      D_{t'}^{j} & =  D_{t}^{j} \quad \text{and} 
      \quad N_{t'}^{j} =  N_{t}^{j}. 
    \end{align*}
  \item[Case 2:] Otherwise, let $I$ denote the index of the particle jumping at
    time $t'$, we use
    \begin{equation}
      \label{eq:updateDNk}
      \begin{cases}
        D_{t'}^{j}  = D_{t}^{j} -f(Y^{I}_{t}); \quad 
        D_{t'}^{j+1}  =  D_{t}^{j+1} + f(Y^{I}_{t'}) &\mbox{ if }
        j=X_{t}^{I}, \\
        N_{t'}^{j}  = N_{t}^{j} - 1; \quad 
        N_{t'}^{j+1}  =  N_{t}^{j+1} + 1 &\mbox{ if }
        j=X_{t}^{I}, \\
        N_{t'}^{j}  =  N_{t}^{j}; \quad D_{t'}^{j}  =  D_{t}^{j}
        &\mbox{ otherwise.}
    \end{cases}
  \end{equation}
\end{description}
One should notice that in cases 2 and 3, the updating cost does not depend on
the size of the vectors  $(D_t^j)_j$ and  $(N_t^j)_j$.  Using these updating
formulas, we can improve Algorithm~\ref{algodetail} to obtain
Algorithm~\ref{algodetail2}.  With this new algorithm we only have to compute
$D$ full approximations of the conditional expectation which has a unit a cost
of $O(M+N)$ and the rest of the time we use the updating
formulas~\eqref{eq:updateDNk}, which happens on average less than $N T
\frac{\bar \lambda  \bar f}{\underline f}$ times. Then, the overall cost of
this new algorithm is $O( D(M+N) + N T \frac{\bar \lambda  \bar f}{\underline
f} )$. For fixed model parameters, this complexity reduces to $O(DN)$. This new
algorithm has a linear cost with respect to the number of particles. Thus, we
managed to propose an interacting particle algorithm with the same cost as a
crude Monte--Carlo method for SDEs since $M$ is in practice fixed and much
smaller than~$N$. The CPU times of the two algorithms are compared in the
following examples.

\subsection{Marginal distributions}

%
\begin{figure}[htp]
  \begin{center}
    \subfigure[Default distribution of the SLI model for $Y$ given by 
    Equation~\eqref{eq:arnsdorf} with $T=1, Y_0=1, a=1, \sigma=0.3,
    \gamma=1, \bar \lambda=2.5$.\label{fig:marginal_arnsdorf}]
    {\includegraphics[width=0.47\textwidth]{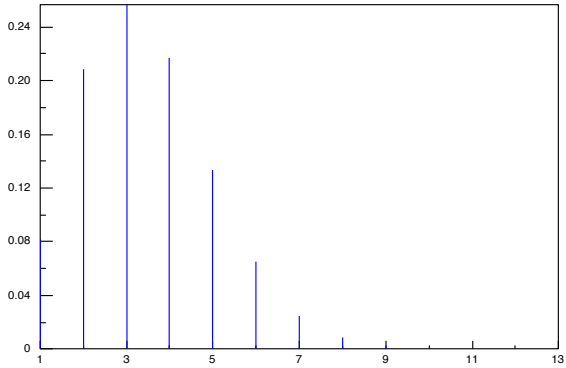}}
    \hfill
    \subfigure[Default distribution of the SLI model for $Y$ given by 
    Equation~\eqref{eq:lopatin} with $T=1, Y_0=1, \kappa=1, \sigma=0.3,
    \bar \lambda=2.5$.\label{fig:marginal_lopatin}]
    {\includegraphics[width=0.47\textwidth]{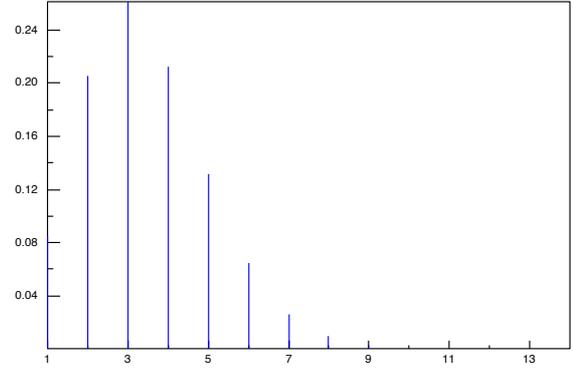}}
    \subfigure[Default distribution of the LI model for 
    $T=1, \bar \lambda=2.5$.]
    {\includegraphics[width=0.47\textwidth]{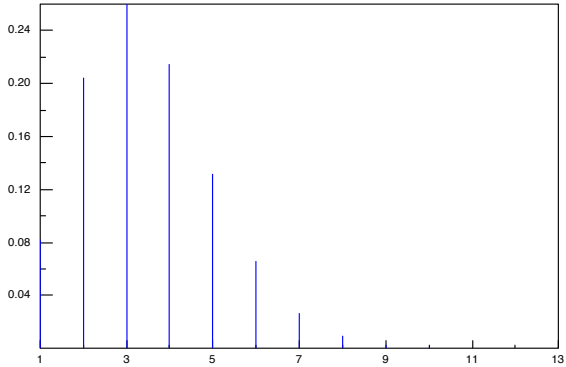}}
    \hfill
    \subfigure[Binomial distribution function with parameters $M$, $p = 1 -
    \expp{- \bar \lambda T / M}$.\label{fig:binomial}]
    {\includegraphics[width=0.47\textwidth]{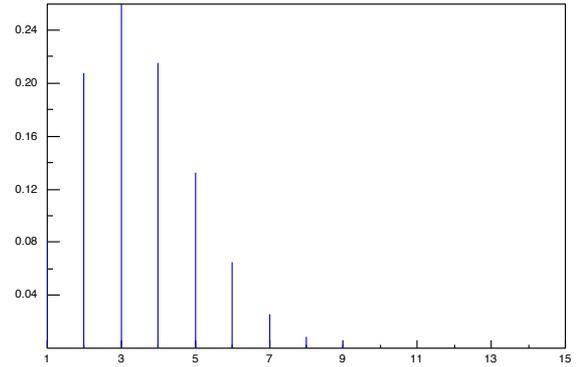}}
  \caption{Comparison of the marginal distributions of the LI and SLI models.
  The simulations use $N=50,000$ samples, $D=100$ and $\underline{f}=1/3, \overline{f}=3$.}
  \label{fig:marginal}
  \end{center}
\end{figure}

In Figures~\ref{fig:marginal}, we draw the probability distribution of $X$ for
the LI and SLI models for both processes $Y$ considered in this part. The
comparison of these graphs highlights how the LI and SLI models effectively
mimic their marginal distributions; their probability distributions look almost
the same. Since the LI model corresponds to independent default times with
intensity~$\frac{\bar \lambda}{M}$, the default distribution at time $T$ is actually the Binomial
law with parameters $M$ and $p=1-\exp\left(-\frac{\bar \lambda}{M}T\right)$ as
we can see on Figure~\ref{fig:binomial}.

\begin{table}[htp]
  \centering
  \begin{tabular}{l|c|c}
    & Algorithm~\ref{algodetail} & Algorithm~\ref{algodetail2}  \\
    \hline
    Model of Fig.~\ref{fig:marginal_arnsdorf} & 239 & 1.51 \\
    \hline
    Model of Fig.~\ref{fig:marginal_lopatin} & 229 & 1.49
  \end{tabular}
  \caption{Comparison of the CPU times (in seconds) of the two algorithms with the same parameters as in Figure~\ref{fig:marginal}.}
  \label{tab:CPU}
\end{table}
We compare in Table~\ref{tab:CPU} the computational times of the two algorithms
and the gain obtained by the second approach is definitely outstanding.
Algorithm~\ref{algodetail2} massively outperforms Algorithm~\ref{algodetail} by
a factor of $150$. Of course, this gain will be all the more important as the number of
particles increases.

Given the impressive match of the marginal distributions, we would like to
numerically investigate the difference between their distributions as processes.
To do so, we have computed in each model the length of the longest
interval during which $X$ does not jump defined by
\begin{align}
  \label{eq:max_constant}
  \tau = \sup \{ t \in [0, T] : \; \exists u \in [0, T-t], \; X_{{u+t}^-} = X_u \} 
\end{align}
Note that with this definition, $\tau = T$ when $X$ does not jump on the
interval $[0,T]$.  The histograms of $\tau$ in the LI and SLI models are shown
in Figure~\ref{fig:max_constant}; we can see that, in the SLI, the
length of the longest interval without jumps for $X$ can be very small with a
probability much higher than in the LI model (the l.h.s. of the histogram in the
SLI model is fatter than in the LI model). This impression is reinforced by more
quantitative observations. From the data used to plot these histograms, we have
computed in Table~\ref{tab:quantiles} several values of the cumulative distribution
function of the length of the longest interval without defaults both in the LI
and SLI models.  These quantities differ sufficiently to be numerically convinced
that these two distributions do not match.
\begin{table}[h!]
  \centering\begin{tabular}{l|c|c}
    & $\P(\tau \le T/4)$ & $\P(\tau \le T/8)$ \\
    \hline
    SLI model & 0.1911 ($\pm$ 0.0033) & 0.0200 ($\pm$ 0.0012)\\
    \hline
    LI model & 0.1645 ($\pm$ 0.0033) & 0.0113 ($\pm$ 0.0009) \\
  \end{tabular}
  \caption{Some values of the distribution function of the length of the longest
  interval without jumps for $T=2$ and $\bar \lambda=2.5$. The SLI model
  is defined as in Figure~\ref{fig:max_constant_sli}. The simulations use
  $50,000$ samples. The values between braces correspond to twice the standard
  deviation of the estimator. } \label{tab:quantiles}
\end{table}

\begin{figure}[h!tp]
  \begin{center}
    \subfigure[SLI model for $Y$ given by 
    Equation~\eqref{eq:arnsdorf} with $Y_0=1, a=1, \sigma=0.3,
    \gamma=3$.]
    {\includegraphics[width=0.48\textwidth]{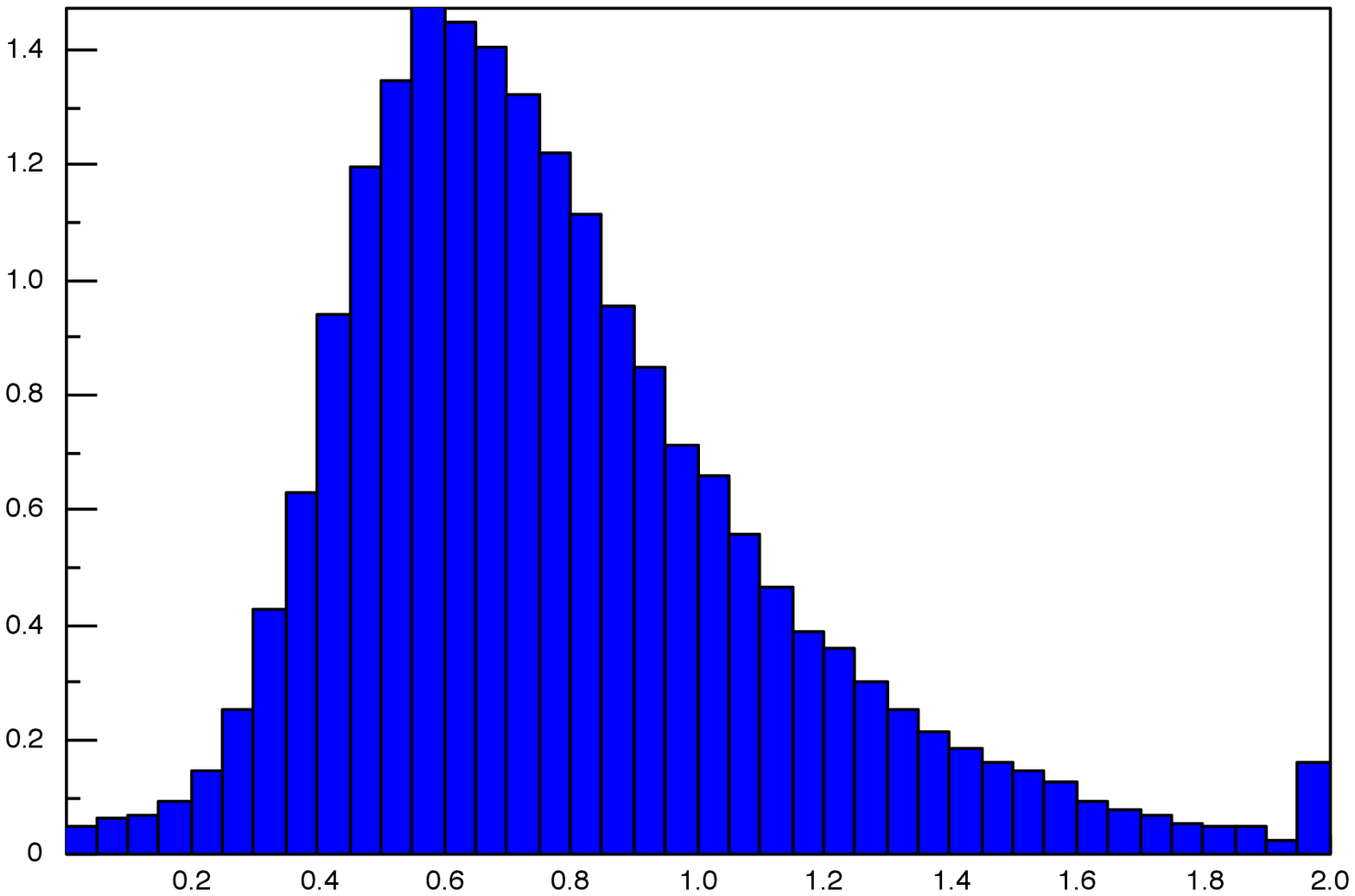}
    \label{fig:max_constant_sli}}
    \hfill
    \subfigure[LI model]{\includegraphics[width=0.495\textwidth]{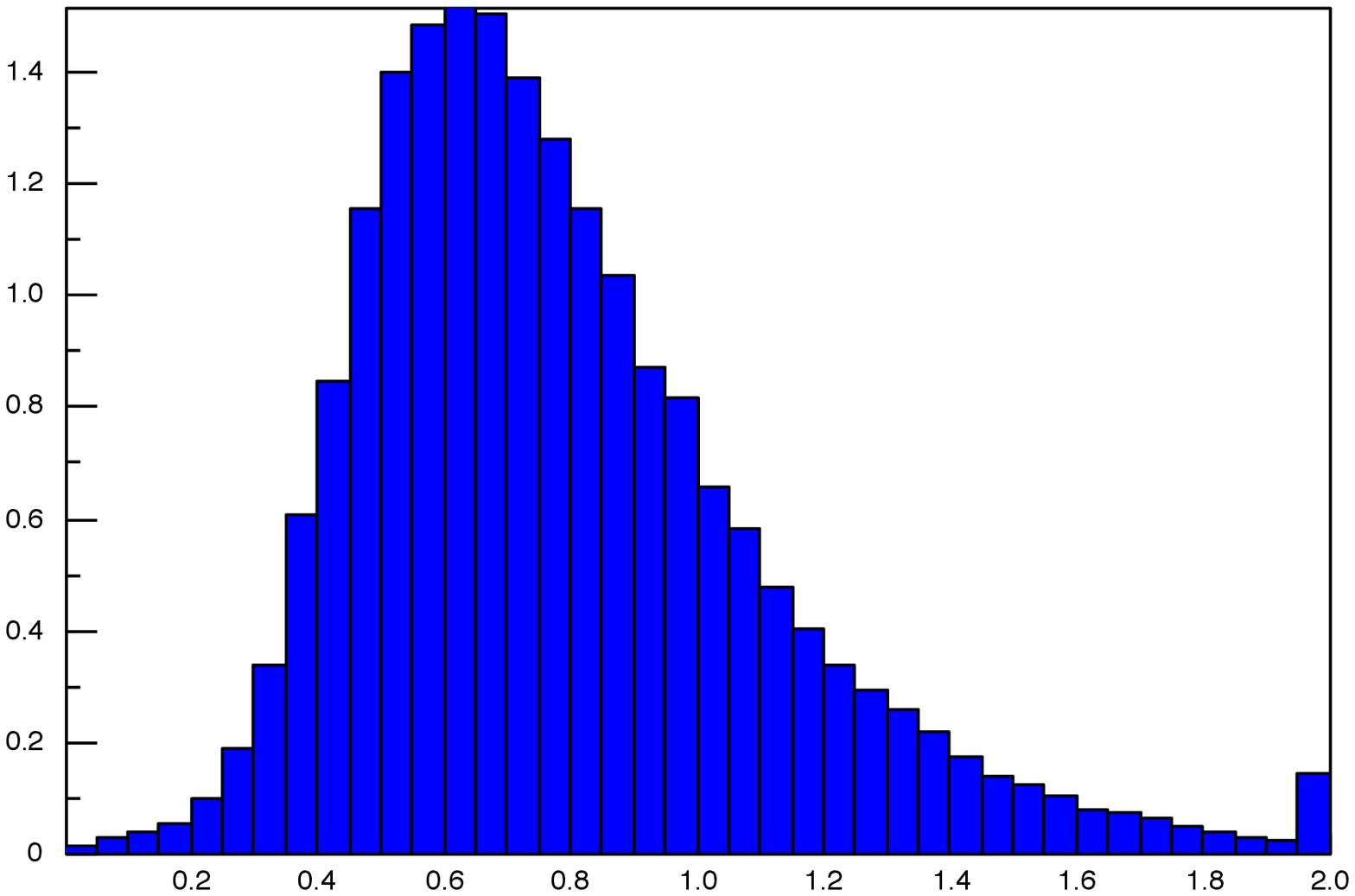}}
    \caption{Histogram of the length of the longest interval without $X$ jumping
    for $T=2$ and $\bar \lambda=2.5$. These histograms uses $50,000$ samples.}
    \label{fig:max_constant}
  \end{center}
\end{figure}

\subsection{Convergence of the interacting particle system}

When introducing the interacting particle system, we emphasized that it was not
only a theoretical tool but that it was also of practical interest as it satisfies
a strong law of large numbers. From a numerical point of view, the efficiency of
the particle system depends on the rate at which every particle converges to the
invariant probability (see Section~\ref{section_speed}). In that section, we
proved that this convergence rate was faster than $N^\alpha$ for $0 < \alpha <
1/2$ where $N$ is the number of particles. Now, we want to study this
convergence rate in several examples: the first example only involves the
marginal distribution of the particle system at maturity time, whereas the 
other two examples require the knowledge of the whole distribution
and not only the marginal ones.

\paragraph{Number of defaults distribution.}
First, we start with a simple example. We want to study the convergence of the
estimator of $\P(X_T = 3)$ computed on the particle system. We ran $5000$
independent copies of the interacting particle systems and we computed the value
of the estimator for each system. In Figure~\ref{fig:3default}, we can
see the centered and renormalised histogram of the values obtained for the
empirical estimators. The histogram can be compared to the density of the standard
normal distribution plotted as a solid line on Figure~\ref{fig:3default}
and they match pretty well. This result suggests that a kind of central limit theorem
should hold in practice even though we did not manage to prove such a result. \\

\begin{figure}[h!tp]
  \begin{center}
    \includegraphics[width=0.8\textwidth]{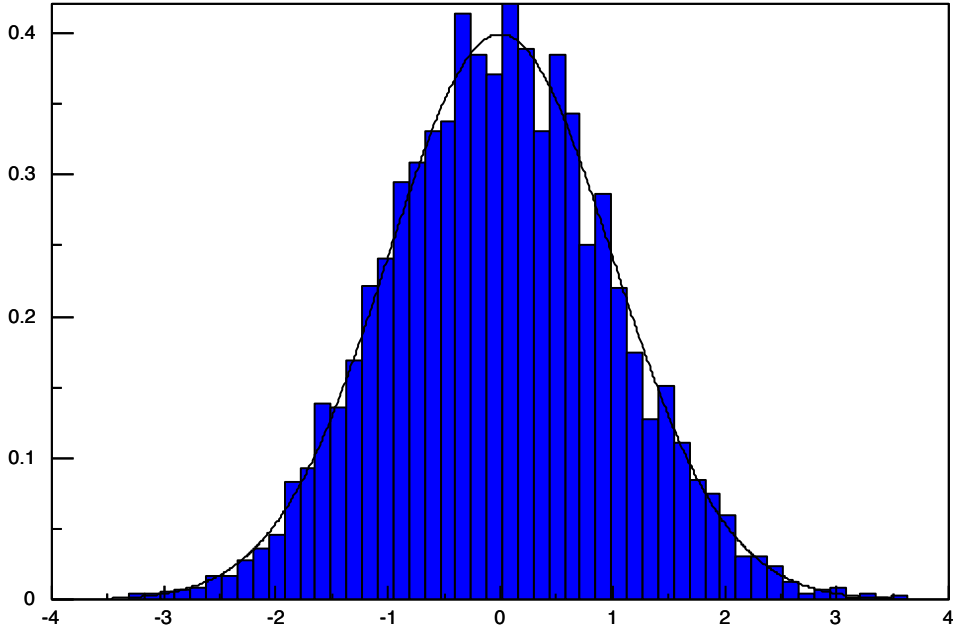}
    \caption{Centered and renormalised distribution of the estimation of $\P(X_T
    = 3)$ using the particle system when $Y$ is given by
    Equation~\eqref{eq:arnsdorf} with $T=1, Y_0=1, a=1, \sigma=0.3,
    \gamma=1, \bar \lambda=2.5$, $\underline{f}=1/3, \overline{f}=3$ and
    $N=10000$. The histogram was obtained using $5000$ independent particle
    systems.} \label{fig:3default}
  \end{center}
\end{figure}

\paragraph{Asian option on the number of defaults.}
For our second example, we consider an Asian option on the default counting
process whose price is given by 
\[
P = \E\left(\frac{1}{T} \int_0^T X_u du - K\right)_+ 
\]
This price $P$ will be approximated using the corresponding particle system
estimate $P_N$, where $N$ is the number of particles.  We are interested in the
limiting distribution of $P_N$. Because the process $X$ has no continuous part
and only makes jumps of size $1$ and $X_0 = 0$, 
the pathwise integral can be rewritten
\[
\frac{1}{T} \int_0^T X_u du = X_T - \frac{1}{T} \sum_{t \le T \; s.t. \; X_{t-} \neq X_t} t 
\]
Hence, there is no need to approximate the integral, it can be computed exactly
(up to the simulation of $X$). The example requires to sample the joint
distribution of $X_T$ and the sum of the default times.

On Figure~\ref{fig:asian_hist}, we have plotted the distribution of $P_N$ after
renormalizing and centering. As before, the solid line is the standard normal
density. We can see that the limiting distribution looks very much like the
Gaussian distribution, but such an histogram does not enable to determine the
rate of convergence to the limiting distribution. Actually, the rate of
convergence is given by the decrease rate of $\sqrt{\Var(P_N)} = \left(\E(|P_N -
P|^2)\right)^{1/2}$. From a practical point of view we do not have access to
$P$, so we have approximated it by the empirical mean $\hat P$ of the
data set used to build the histogram of Figure~\ref{fig:asian_hist}.

We can see on Figure~\ref{fig:asian_cv_rate} that the rate of decrease of
$\sqrt{\E(|P_N - \hat P|^2)}$ recalls the shape of a negative power function.
Then, we have decided to compute the linear regression of $-\frac{1}{2} \log
\E(|P_N - \hat{P}|^2)$ with respect to $\log(N)$ on our simulations of $P_N$ for $N$
varying from $100$ to $10,000$ with a step size of $100$, which gives a set of
$100$ data. The idea of the regression is to write
\[
- \frac{1}{2} \log \E(|P_N - \hat{P}|^2) = \alpha \log(N) + \beta + \varepsilon_N
\]
and to minimize the series $\sum_{N} \varepsilon_N^2$.  The minimum is achieved
for $\alpha = 0.5014$, $\beta=-0.1361$ and the empirical variance of the
sequence $(\varepsilon_n)$ is equal to $10^{-4}$. This computation yields that
the rate of convergence to the limiting distribution is $\sqrt{N}$. It ensues
from this result combined with the analysis of the
histogram~\ref{fig:asian_hist} that a Central Limit Theorem with rate
$\sqrt{N}$ should hold.

\begin{figure}[h!tp]
  \begin{center}
    \includegraphics[width=0.8\textwidth]{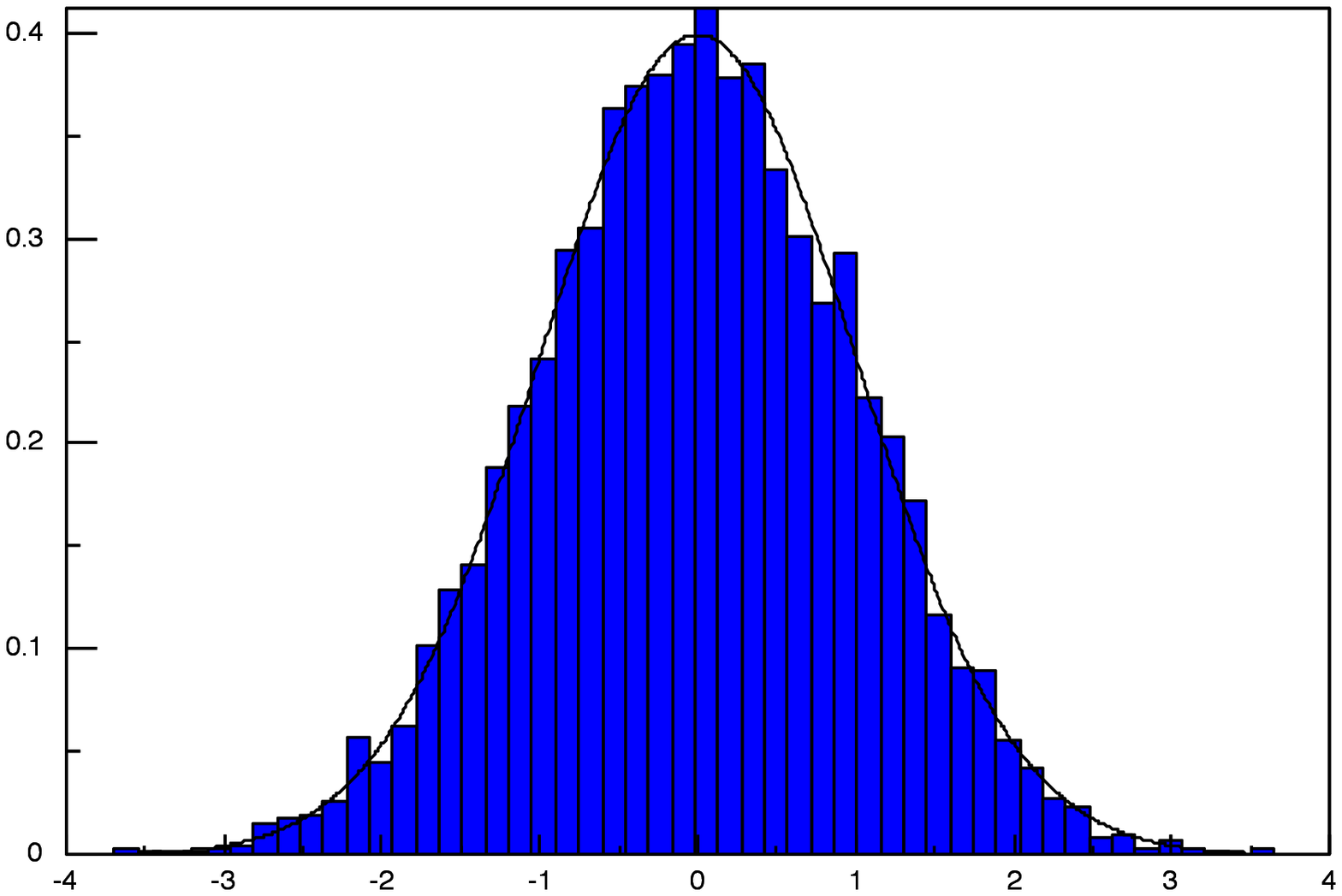}
    \caption{Centered and renormalised distribution of the estimator $P_N$ of
    the Asian option price using the particle system when $Y$ is given by
    Equation~\eqref{eq:lopatin} with $T=2, Y_0=1, a=1, \sigma=0.3, \kappa=1$,
    $\underline{f}=1/3, \overline{f}=3$ and $N=10000$. The histogram was
    obtained using $10,000$ independent particle systems.} \label{fig:asian_hist}
  \end{center}
\end{figure}

\begin{figure}[h!tp]
  \begin{center}
    \includegraphics[width=0.8\textwidth]{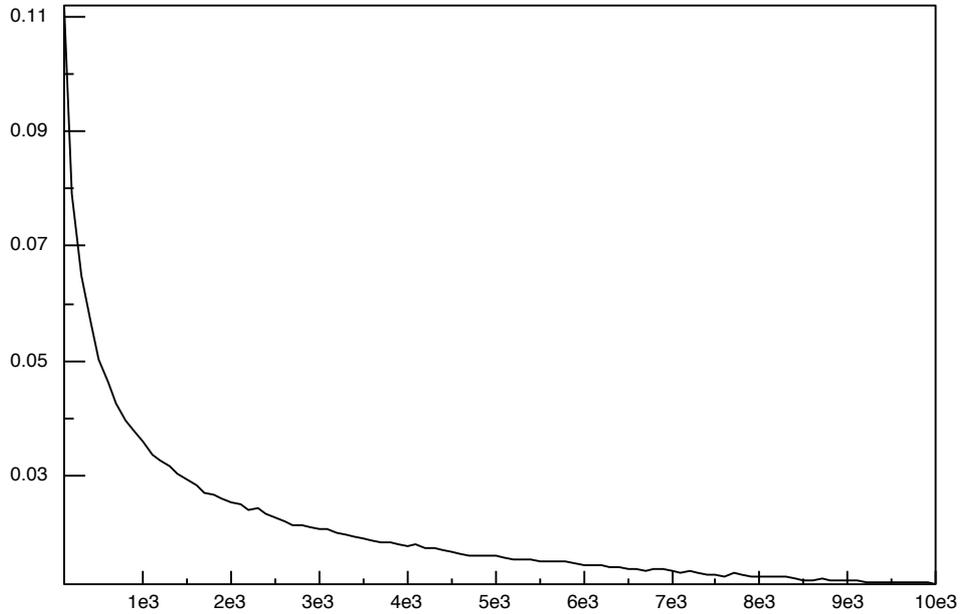} 
    \caption{Convergence rate of $\sqrt{\E|P_N - \hat P|^2}$ w.r.t the number of
    particles $N$.}
    \label{fig:asian_cv_rate}
  \end{center}
\end{figure}

\paragraph{Longest interval without jump.} 

In this paragraph, we are interested in the convergence rate of the estimator of
the length of the longest time interval with no jump. We recall the quantity of
interest already defined in Equation~\eqref{eq:max_constant}
\begin{align*}
  \tau = \sup \{ t \in [0, T] : \; \exists u \in [0, T-t], \; X_{{u+t}^-} = X_u \} 
\end{align*}
and we consider its particle system estimator $\tau_N$. To numerically
sample from the distribution of $\tau$, we need to know the joint distribution
of the jumping times of $X$, which is a prime case of a pathwise estimator.

On Figure~\ref{fig:max_constant_hist}, we can see  the centered and renormalized
distribution of $\tau_N$ together with the standard normal density function
plotted as a solid line.  Again, the limiting distribution looks very much like
a Gaussian distribution. Using these $5,000$ independent particle systems, each
with $10,000$ particles, we can compute an approximation of $\tau$, denote
$\hat \tau$ in the following. Now, we can run independent simulations of
particle systems with a number of particles $N$ varying from $100$ to $10,000$.
We study the rate of decrease of $\sqrt{\E|\tau_N - \hat \tau|^2}$, which
according to Figure~\ref{fig:max_constant_cv_rate} shows a negative power
function shape. If we linearly regress $\log \sqrt{\E|\tau_N - \hat \tau|^2}$
against $\log(N)$ we find that we can write
\[
- \frac{1}{2} \log \E(|\tau_N - \hat{\tau}|^2) = \alpha \log(N) + \beta + \varepsilon_N
\]
The linear regression yields $\alpha = 0.4865$, $\beta=1.016$ and the empirical
variance of the sequence $(\varepsilon_n)$ is equal to $0.0004$. This regression
yields that the rate of convergence to the limiting distribution is $\sqrt{N}$,
which focuses the existence of a Central Limit Theorem with rate $\sqrt{N}$.

\begin{figure}[h!tp]
  \begin{center}
    \includegraphics[width=0.8\textwidth]{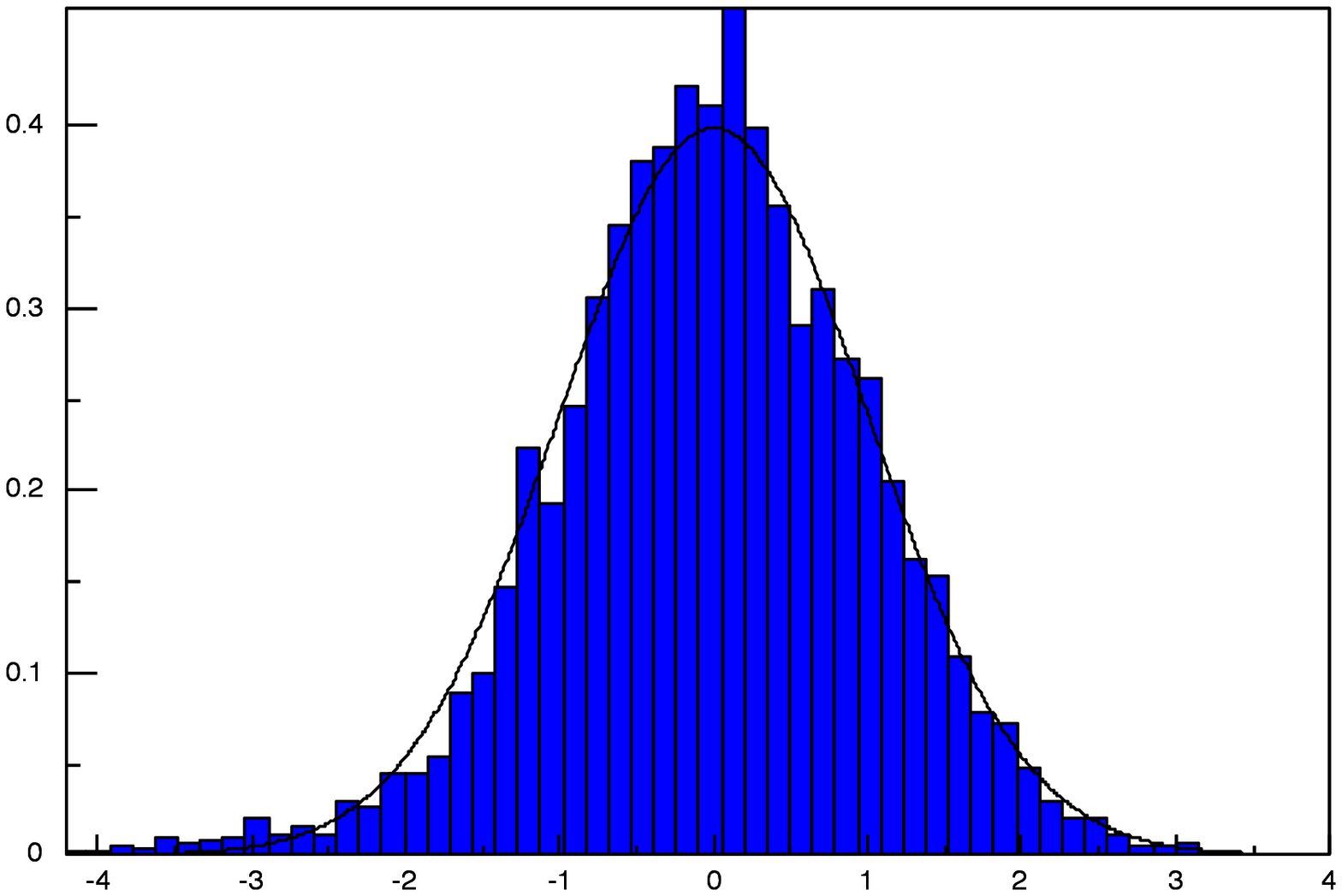}
    \caption{Centered and renormalised distribution of the estimator $\hat \tau$
    of $\tau$ using the particle system when $Y$ is given by
    Equation~\eqref{eq:lopatin} with $T=2, Y_0=1, a=1, \sigma=0.3, \kappa=1$,
    $\underline{f}=1/3, \overline{f}=3$ and $N=10,000$. The histogram was
    obtained using $10,000$ independent particle systems.}
    \label{fig:max_constant_hist}
  \end{center}
\end{figure}

\begin{figure}[h!tp]
  \begin{center}
    \includegraphics[width=0.8\textwidth]{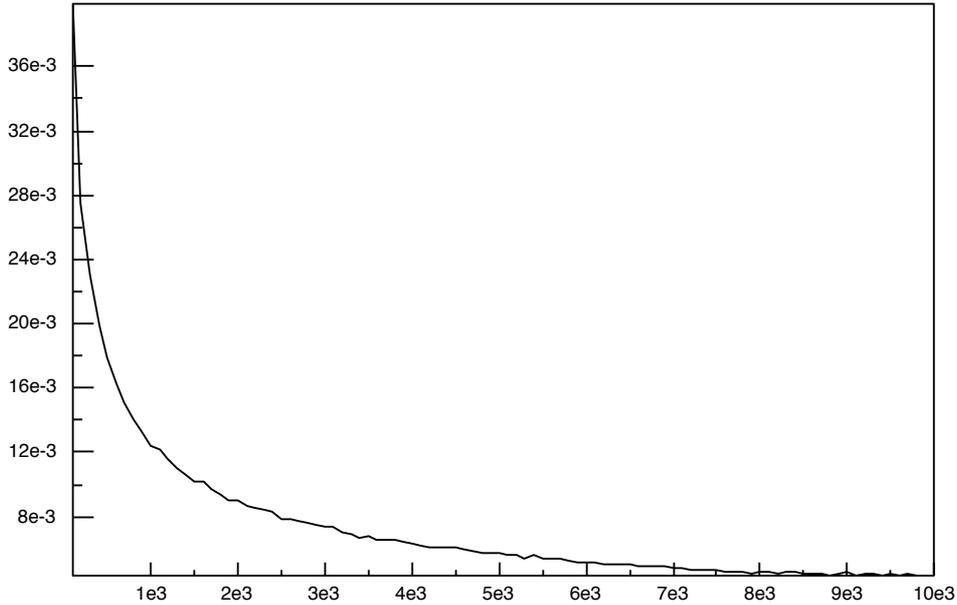} 
    \caption{Convergence rate of $\sqrt{\E|\tau_N - \hat \tau|^2}$ w.r.t the number of
    particles $N$.}
    \label{fig:max_constant_cv_rate}
  \end{center}
\end{figure}

\section{Conclusion}

Local intensity models are wide spread for modelling the default counting
process in credit risk. Recently, more sophisticated models with a stochastic
factor involved in the intensity have been introduced in the literature. These
stochastic local intensity models can be automatically calibrated to CDO
tranche prices by properly choosing the local part of the intensity. This
particular choice of the local intensity gives rise to a very specific family
of SLI dynamics for which we have investigated the existence and uniqueness of
solutions. This theoretical study has been carried out using particle systems,
which turned out to be a clever tool for the numerical simulation of such
dynamics. We have proved that particle Monte-Carlo algorithms based on this
particle system approach almost surely converge. The theoretical study of the
convergence rate enabled us to prove that the almost sure convergence took
place at a rate faster that $N^\alpha$ for any $0 < \alpha < 1/2$. Obtaining a
Central Limit Theorem type result for such particle systems remains an open
question, even though we could highlight such a behaviour in all our
simulations. Last, we have shown that the interacting particle system can be
sampled with a computational cost in $O(ND)$, which is the same asymptotic cost as
a Monte-Carlo algorithm for standard SDEs. 

\clearpage
\appendix
\section{Proof of Proposition~\ref{prop_FP}}\label{App_prop_FP}
The scheme of the proof is the following:
  \begin{itemize}
  \item First, we prove that $\Psi_+:t,x\longmapsto \Psi(t,x_+)$ is globally
    Lipschitz in $x$. Then, $(\cE'')$ admits a
    unique solution on $\R_+$.
  \item Second, we prove that the solution satisfies
    $\forall t \ge 0, \; P(t) \ge 0$ and $|P(t)|=1$.
  \end{itemize}

{\bf Step 1: $\Psi_+$ is globally Lipschitz.}
    Let us prove that for all $(x,y) \in E \times E$ there
  exists a constant $K$ such that $|\Psi_+(t,x)-\Psi_+(t,y)|\le K|x-y|$. We
  have
  \begin{align*} |\Psi_+(t,x)-\Psi_+(t,y)|=\sum_{i \in \cL_M} \sum_{j \ge 0}
    |(\Psi(t,x_+))^i_j-(\Psi(t,y_+))_j^i|.
  \end{align*} Bounding this
  quantity boils down to  bound $\displaystyle \sum_{i \in \cL_M} \sum_{j \ge
    0} \left((A_1)^i_j + 2
  f(j)\lambda(t,i)(A_2)^i_j\right)$ where
  \begin{align*}
    & (A_1)^i_j:=|\sum_{k \ge 0} \mu^i_{kj}
    ((x^i_k)_+-(y^i_k)_+)|,\\
    &(A_2)^i_j:=\left|\frac{\sum_{l=0}^{\infty}(x^{i}_l)_+}{\sum_{l=0}^{\infty}f(l)(x^{i}_l)_+}(x^i_j)_+-\frac{\sum_{l=0}^{\infty}(y^{i}_l)_+}{\sum_{l=0}^{\infty}f(l)(y^{i}_l)_+}(y^i_j)_+\right|.\\
  \end{align*}
  
  {\bf Bound for $A_1$.} We easily get from Hypothesis~\ref{hypo:discrete}
  \begin{align*}
   \sum_{i \in \cL_M} \sum_{j \ge
    0}  (A_1)^i_j \le \sum_{i \in \cL_M} \sum_{k \ge 0} \sum_{j \ge 0}
    |\mu^i_{kj}| |(x^i_k)_+-(y^i_k)_+| &= \sum_{i \in \cL_M} \sum_{k \ge 0} 2
    |\mu^i_{kk}| |(x^i_k)_+-(y^i_k)_+| \\&\le 2 \sup_{i\in \cL_M ,k \in \N}
  |\mu^i_{kk}|  |x-y|.
\end{align*}

{\bf Bound for $A_2$.} To
bound it, we introduce
  $\pm \frac{\sum_{l=0}^{\infty}(y^i_l)_+}{\sum_{l=0}^{\infty}f(l)(x^i_l)_+}(x^i_j)_+$
  and $\pm
  \frac{\sum_{l=0}^{\infty}(y^i_l)_+}{\sum_{l=0}^{\infty}f(l)(y^i_l)_+}(x^i_j)_+$
  in order to split $(A_2)^i_j$ in three terms. Each of them is bounded in the
  following way
  \begin{align*}
    \left|\frac{\sum_{l=0}^{\infty}(x^{i}_l)_+}{\sum_{l=0}^{\infty}f(l)(x^{i}_l)_+}(x^i_j)_+-\frac{\sum_{l=0}^{\infty}(y^i_l)_+}{\sum_{l=0}^{\infty}f(l)(x^i_l)_+}(x^i_j)_+\right|&\le
    \frac{(x^i_j)_+}{\sum_{l=0}^{\infty}f(l)(x^{i}_l)_+}\sum_{l=0}^{\infty}|(x^{i}_l)_+-(y^i_l)_+|,\\
   \left|\frac{\sum_{l=0}^{\infty}(y^i_l)_+}{\sum_{l=0}^{\infty}f(l)(x^i_l)_+}(x^i_j)_+-\frac{\sum_{l=0}^{\infty}(y^i_l)_+}{\sum_{l=0}^{\infty}f(l)(y^i_l)_+}(x^i_j)_+\right|&\le
   \frac{\overline{f}}{\underline{f}}\frac{(x^i_j)_+}{\sum_{l=0}^{\infty}f(l)(x^{i}_l)_+}\sum_{l=0}^{\infty}|(x^{i}_l)_+-(y^i_l)_+|,\\
   \left|\frac{\sum_{l=0}^{\infty}(y^i_l)_+}{\sum_{l=0}^{\infty}f(l)(y^i_l)_+}(x^i_j)_+-
     \frac{\sum_{l=0}^{\infty}(y^i_l)_+}{\sum_{l=0}^{\infty}f(l)(y^i_l)_+}(y^i_j)_+\right|&\le \frac{1}{\underline{f}}|(x^{i}_j)_+-(y^i_j)_+|.
 \end{align*}
 Then, $\sum_{i \in \cL_M} \sum_{j \ge 0} f(j) \lambda(t,i)(A_2)^i_j \le \overline{\lambda}(1+2\frac{\overline{f}}{\underline{f}})|x-y|$.
 Combining bounds on $A_1$ and $A_2$, we get
 $|\Psi_+(t,x)-\Psi_+(t,y)| \le K |x-y|$, where $\displaystyle K=2 \sup_{i\in \cL_M ,k \in \N}
  |\mu^i_{kk}| +2\overline{\lambda}\left(1+2\frac{\overline{f}}{\underline{f}}\right)$.\\

 {\bf Step 2: the solution is positive with norm $1$.}  Let $P(t)$ denote the unique
 solution of $(\mathcal{E}'')$.\\
 First, we prove that $\forall t \in \R_+, \forall (i,j) \in
 \mathcal{L}_M\times \N$, $P^i_j(t) \ge 0$.\\
 $P^i_j(t)$ satisfies
 $$
 \left\{\begin{array}{l}
     (P^i_j)'(t)=\sum_{k\neq j} \mu^i_{kj}
     (P^i_k(t))_+
     +\ind{i\ge
       1}\frac{\lambda(t,i-1)f(j)}{\varphi((P(t))_+,i-1)}(P^{i-1}_j(t))_+
     -\left(\frac{\lambda(t,i)f(j)\ind{i \le M-1}}{\varphi((P(t))_+,i)}-\mu^i_{jj}\right)(P^i_j(t))_+,\\
      P^i_j(0)=\delta_{ix_0}\delta_{jy_0}.
     \end{array}\right.
 $$
 We assume that there exists $\overline{t} \ge 0$ such that
   $P^i_j(\overline{t})<0$. We also introduce $u:=\sup\{s \le \ovt :
   P^i_j(s)=0\}$. Then, we integrate the above equation between $u$ and
   $\ovt$. We get
   \begin{align*}
     (P^i_j)(\ovt)=\int_u^{\ovt} \sum_{k\neq j} \mu^i_{kj}
     (P^i_k(s))_+ +\ind{i\ge
       1}\frac{\lambda(s,i-1)f(j)}{\varphi((P(s))_+,i-1)}(P^{i-1}_j(s))_+
     ds.
   \end{align*}
   The l.h.s. is strictly negative whereas the r.h.s. is non negative. Then,
   $P^i_j(t) \ge 0$ for all $t \in \R_+$.\\
 Second, we prove  $\forall t \in \R_+$, $ |P(t)|=1$. Since $P$ is
 non negative, $(|P(t)|)'=\sum_{i\in \cL_M} \sum_{j \ge 0} (P^i_j)' (t)$. Moreover,
 $\sum_{i \in \cL_M} \sum_{j \ge 0} (P^i_j)'(t)=0$. Then,  $|P(t)|=|P(0)|=1$.

\section{Proof of Proposition~\ref{unicite_eds_sauts}}\label{app_proof_unicite_eds}
In fact, we can explicitly describe the solution of~\eqref{EDS_avec_sauts} as follows. We have $(X_0,Y_0)=(x_0,y_0)$. Then, up to the first jump of~$X$, $Y_t$ is necessarily defined as the unique strong solution of
the following SDE
\begin{equation*}
Y_t=y_0+\int_0^t b(s,x_0,Y_s)ds + \int_0^t \sigma(s,x_0,Y_s)dW_s, \ t\le T.
\end{equation*}
The process $(X_t,t\ge0)$ only increases by jumps of size~$1$, and has at most
$M$ jumps since $\lambda(t,M)=0$. These jumps may only occur at the times $T^n$
and a jump do occur at time~$T^n$ if the following condition is satisfied
$$U^n \le \frac{\underline{f}}{\overline{\lambda}\overline{f}}
\frac{\lambda(T^n-,X_{T^n-})f(Y_{T^n-})}{\varphi^{\Q}(T^n-,X_{T^n-})}. $$ 
Observe here that for any $x\in \cL_M$, $t\mapsto\varphi^{\Q}(t,x)$ is
càdlàg since the process $(X,Y)$ is càdlàg under~$\Q$, and
$\varphi^{\Q}(T^n-,X_{T^n-})$ is well defined.  If the jump occurs, we have
$$Y_{T^n}=Y_{T^n-}+\gamma(T^n-,X_{T^n-},Y_{T^n-}), $$ and $Y_t$ is defined up
to the next jump of~$X$ as the unique strong solution of the SDE $$Y_t=
Y_{T^n}+\int_{T_n}^t b(s,X_{T_n},Y_s)ds + \int_{T_n}^t
\sigma(s,X_{T_n},Y_s)dW_s, \  T_n \le t \le T.$$

\section{Proof of $\E[\sup_{t \in [0,T]} |Y^{1,N}_t|^p] < \infty$}\label{sect:moment_Y}
We assume that $Y^{1,N}$ satisfies (\ref{eq4}) and $X^{1,N}$ is a Poisson
process with intensity (\ref{eq3}). Then
\begin{align*}
  |Y^{1,N}_t|^p &\le 4^{p-1} \left( y_0^p+ \left|\int_0^t b(s,X^{1,N}_s,Y^{1,N}_s)ds\right|^p + \left|\int_0^t
      \sigma(s,X^{1,N}_s,Y^{1,N}_s)dW_s\right|^p \right. \\
    &\ \ +\left. \left|\int_0^t
      \gamma(s^-,X^{1,N}_{s^-},Y^{1,N}_{s^-})dX^{1,N}_s\right|^p\right) \\
 &\le 4^{p-1} \left( y_0^p+ T^{p-1} \int_0^t
   \left|b(s,X^{1,N}_s,Y^{1,N}_s)\right|^p ds + \sup_{u \le t}\left| \int_0^u
      \sigma(s,X^{1,N}_s,Y^{1,N}_s)dW_s\right|^p \right. \\ 
 &\ \ +\left. M^{p-1}\int_0^t
      |\gamma(s^-,X^{1,N}_{s^-},Y^{1,N}_{s^-})|^pdX^{1,N}_s\right) ,
\end{align*}
since~$X^{1,N}$ jumps at most $M$~times. The r.h.s being increasing w.r.t~$t$,
we can replace in the l.h.s. $|Y^{1,N}_t|^p$ by $\sup_{u \le t}|Y^{1,N}_u|^p$.
Burkholder-Davis-Gundy inequality yields to
$$\E\left[\left|\sup_{u \le t} \int_0^u
      \sigma(s,X^{1,N}_s,Y^{1,N}_s)dW_s\right|^p  \right] \le C_p T^{p/2-1} \E
  \left[ \int_0^t \left|\sigma(s,X^{1,N}_s,Y^{1,N}_s)\right|^p ds  \right].$$
On the other hand, $\E \left[ \int_0^t
      |\gamma(s^-,X^{1,N}_{s^-},Y^{1,N}_{s^-})|^pdX^{1,N}_s \right] \le
   \frac{\overline{\lambda} \overline{f}}{\underline{f}} \int_0^t\E
   \left[|\gamma(s^-,X^{1,N}_{s^-},Y^{1,N}_{s^-})|^p\right] ds$ since the jump
   intensity of~$X^{1,N}$ is upper bounded by~$\frac{\overline{\lambda}
   \overline{f}}{\underline{f}}$. Now since $b,\ \sigma,$
and~$\gamma$ have a sub linear growth with respect to~$y$ (see Hypothesis ~\eqref{hypo:continuous}), we get
\begin{align*}
  \E[\sup_{t \le T} |Y^{1,N}_t|^p] \le C\left(\int_0^T 1+ \E[\sup_{u \le s} |
    Y^{1,N}_u|^p] ds\right) ,
\end{align*}
which gives the result by Gronwall's lemma.

\section{Proof of Lemma~\ref{pi_tendue}}\label{App_tension}

The proof of this lemma is done in two steps. First, we claim that
$(\pi^N)_N$ is tight if and only if the sequence
$((X^{1,N}_t,Y^{1,N}_t)_{t\in[0,T]})_N$ is tight. To check this, we first
notice that $\cP(E)$ is a Polish space. By  Proposition~4.6
in~\cite{Meleard},  $(\pi^N)_N$ is tight if and only if
$(I(\pi^N))_N$ is tight. Then Prohorov's Theorem (tightness is
equivalent to sequential compactness) gives with equation~\eqref{eq_pi_N} the
claim.

Now, we must show that $((X^{1,N}_t,Y^{1,N}_t)_{t\in[0,T]})_N$ is tight. We use
Aldous' criterion.

First, we have to check that, for any $\varepsilon
>0$, there exists a constant $K$ such that $\P(\sup_{t\in[0,T]}
|X^{1,N}_t|+|Y^{1,N}_t|>K)\le \varepsilon$. This is trivial since $X^{1,N}_t$
is bounded and $\E[\sup_{t\in[0,T]}|Y^{1,N}_t|]<\infty$ (see Appendix \ref{sect:moment_Y}).
 Second, we have to check that for any $\varepsilon >0$,
 $\eta>0$, there exist $\delta>0$ and $n_0$ such that
$$\sup_{n\ge n_0} \sup_{\tau_1,\tau_2 \in \cT_{[0,T]}, \tau_1 \le
  \tau_2 \le \tau_1+\delta}
\P(|(X^{1,N}_{\tau_2},Y^{1,N}_{\tau_2})-(X^{1,N}_{\tau_1},Y^{1,N}_{\tau_1})|>\eta)<\varepsilon,$$
where $\cT_{[0,T]}$ denotes the set of stopping times taking values
in~$[0,T]$. We take $|(x,y)|=\max(|x|,|y|)$ and we assume without loss of
generality that $0<\eta<1$. For convenience, we introduce $\nu^{i,N}_t= \sum_{k=1}^{\infty} \ind{T^{i,k} \le
   t}$: this is a Poisson process with intensity
 $\frac{\overline{\lambda}\overline{f}}{\underline{f}}$. We distinguish
 the two cases: $\nu^{1,N}_{\tau_2}\ge\nu^{1,N}_{\tau_1}+1$ ($X^{1,N}$ jumps in
 $(\tau_1,\tau_2]$ ) and $\nu^{1,N}_{\tau_2}=\nu^{1,N}_{\tau_1}$ ($X^{1,N}$
 does not jump  in  $(\tau_1,\tau_2]$ ). We have:
\begin{align*}
  \P(|(X^{1,N}_{\tau_2},Y^{1,N}_{\tau_2})-(X^{1,N}_{\tau_1},Y^{1,N}_{\tau_1})|>\eta)\le&\P\left(\nu^{1,N}_{\tau_2}\ge\nu^{1,N}_{\tau_1}+1\right) \\
&+\P\left (|Y^{1,N}_{\tau_2}-Y^{1,N}_{\tau_1}|>\eta,
\nu^{1,N}_{\tau_2}=\nu^{1,N}_{\tau_1}\right):=P_1+P_2.
\end{align*}
Since $\P(\nu^{1,N}_{\tau_2}>
\nu^{1,N}_{\tau_1}) \le \E[\nu^{1,N}_{\tau_2}- \nu^{1,N}_{\tau_1}]=\frac{\overline{\lambda}
  \overline{f}}{\underline{f}}\E[\tau_2-\tau_1] $, we get $P_1\le \delta \frac{\overline{\lambda}
  \overline{f}}{\underline{f}}.$\\

$P_2$ is bounded by $\P(\int_{\tau_1}^{\tau_2}
b(s,X_s,Y_s) ds +\int_{\tau_1}^{\tau_2} \sigma(s,X_s,Y_s) dW_s >
\eta)$. Using Markov's inequality, we get
\begin{align*}
  P_2 \le \frac{1}{\eta^2}\E\left[\left|\int_{\tau_1}^{\tau_2}
b(s,X^{1,N}_s,Y^{1,N}_s) ds +\int_{\tau_1}^{\tau_2} \sigma(s,X^{1,N}_s,Y^{1,N}_s) dW_s\right|^2\right].
\end{align*}
Moreover
\begin{align*}
 & \E\left[\left|\int_{\tau_1}^{\tau_2}
b(s,X^{1,N}_s,Y^{1,N}_s) ds +\int_{\tau_1}^{\tau_2} \sigma(s,X^{1,N}_s,Y^{1,N}_s) dW_s\right|^2\right] \\
& \qquad \le
2\left(\E\left[\left|\int_{\tau_1}^{\tau_2}b(s,X^{1,N}_s,Y^{1,N}_s) ds \right|^2\right]
+ \E\left[\left|\int_{\tau_1}^{\tau_2} \sigma(s,X^{1,N}_s,Y^{1,N}_s) dW_s\right|^2\right]\right).
\end{align*}
On the one hand, we have $\left|\int_{\tau_1}^{\tau_2}b(s,X^{1,N}_s,Y^{1,N}_s)
  ds \right|^2\le \delta \int_{\tau_1}^{\tau_2} C(1+
|Y^{1,N}_s|^2 )ds$ for some constant~$C>0$ by Hypothesis~\ref{hypo:continuous}.
On the other hand, Burkholder-Davis-Gundy's inequality gives $\E\left[|\int_{\tau_1}^{\tau_2}
\sigma(s,X^{1,N}_s,Y^{1,N}_s) dW_s|^2\right]\le C \E \left[ \int_{\tau_1}^{\tau_2} 1+
|Y^{1,N}_s|^2 ds \right]$. Since $\E[\sup_{t\le T} |Y^{1,N}_t|^2] < \infty$, we get
$P_2 \le  \frac{C}{\eta^2}\delta$. Thus, combining the upper bounds on $P_1$ and $P_2$, Aldous' criterion is satisfied for $\delta:=\frac{\varepsilon}{\frac{\overline{\lambda}\overline{f}}{\underline{f}}+\frac{C}{\eta^2}}$.

\bibliographystyle{abbrvnat}
\bibliography{biblio}

\end{document}